\documentclass[3p,times]{elsarticle}

\usepackage{ecrc}

\volume{00}

\firstpage{1}

\journalname{}

\runauth{}

\jid{procs}

\jnltitlelogo{}

\usepackage{amssymb}

\usepackage[figuresright]{rotating}

\usepackage{amsmath, epsfig, graphicx, bm}
\usepackage{multirow} 
\usepackage{cases}
\usepackage{float, subfigure}
\usepackage{color, xcolor}
\usepackage{amsthm}
\usepackage{amsfonts}
\usepackage{url}
\usepackage[all,cmtip]{xy}
\usepackage{color,hyperref}
\usepackage[noend]{algpseudocode}
\usepackage{algorithmicx,algorithm}
\usepackage{setspace}
\usepackage[misc]{ifsym}

\definecolor{darkblue}{rgb}{0,0,1}
\hypersetup{colorlinks,breaklinks,linkcolor=darkblue,urlcolor=darkblue,anchorcolor=darkblue,citecolor=darkblue}
\usepackage[normalem]{ulem} 

\newtheorem{theorem}{Theorem}
\newtheorem{proposition}{Proposition}
\newtheorem{lemma}{Lemma}

\newtheorem{assumption}{Assumption}

\newtheorem{example}{Example}

\def \qed {\hfill \vrule height6pt width 6pt depth 0pt}

\allowdisplaybreaks[3] 

\begin{document}

\begin{frontmatter}



\dochead{}

\title{Stochastic Alternating Direction Method of Multipliers \\ for Byzantine-Robust Distributed Learning}


\author{Feng Lin$^1$ \quad Weiyu Li$^2$ \quad  Qing Ling$^{1, 3, 4}$ }

\address{
$1$. Sun Yat-Sen University \\
$2$. University of Science and Technology of China \\
$3$. Guangdong Province Key Laboratory of Computational Science \\
$4$. Pazhou Lab \\
\vspace{0.5em} Corresponding Author: Qing Ling \\
E-mail address: lingqing556@mail.sysu.edu.cn \\
}


\begin{abstract}
This paper aims to solve a distributed learning problem under Byzantine attacks. In the underlying distributed system, a number of unknown but malicious workers (termed as Byzantine workers) can send arbitrary messages to the master and bias the learning process, due to data corruptions, computation errors or malicious attacks. Prior work has considered a total variation (TV) norm-penalized approximation formulation to handle the Byzantine attacks, where the TV norm penalty forces the regular workers' local variables to be close, and meanwhile, tolerates the outliers sent by the Byzantine workers. To solve the TV norm-penalized approximation formulation, we propose a Byzantine-robust stochastic alternating direction method of multipliers (ADMM) that fully utilizes the separable problem structure. Theoretically, we prove that the proposed method converges to a bounded neighborhood of the optimal solution at a rate of $O(1/k)$ under mild assumptions, where $k$ is the number of iterations and the size of neighborhood is determined by the number of Byzantine workers. Numerical experiments on the MNIST and COVERTYPE datasets demonstrate the effectiveness of the proposed method to various Byzantine attacks.
\end{abstract}


\end{frontmatter}


\section{Introduction}
\label{sec:intro}

Most of the traditional machine learning algorithms require to collect training data from their owners to a single computer or data center, which is not only communication-inefficient but also vulnerable to privacy leakage \cite{Agrawal2000, Sicari2015, Lu2018}. With the explosive growth of the big data, federated learning has been proposed as a novel privacy-preserving distributed machine learning scheme, and received extensive research interest recently \cite{Konecny2015, Konecny2016, Kairouz2021}. In federated learning, the training data are stored at distributed workers, and the workers compute their local variables using local training data, under the coordination of a master. This scheme effectively reduces the risk of data leakage and protects privacy.


However, federated learning still faces significant security challenges. Some of the distributed workers, whose identities are unknown, could be unreliable and send wrong or even malicious messages to the master, due to data corruptions, computation errors or malicious attacks. To characterize the worse-case scenario, we adopt the Byzantine failure model, in which the Byzantine workers are aware of all information of the other workers, and able to send arbitrary messages to the master \cite{Lamport1982, Lynch996,Venpaty2013,Chen2018a}. In this paper, we aim at solving the distributed learning problem under Byzantine attacks that potentially threat federated learning applications.

\textbf{Related works}. With the rapid popularization of federated learning, Byzantine-robust distributed learning has become an attractive research topic in recent years. Most of the existing algorithms modify distributed stochastic gradient descent (SGD) to the Byzantine-robust variants. In the standard distributed SGD, at every iteration, all the workers send their local stochastic gradients to the master, while the master averages all the received stochastic gradients and updates the optimization variable. When Byzantine workers are present, they can send faulty values other than true stochastic gradients to the master so as to bias the learning process. 
It is shown that the standard distributed SGD with mean aggregation is vulnerable to Byzantine attacks \cite{Blanchard2017}.

When the training data are independently and identically distributed (i.i.d.) at the workers, the stochastic gradients of the regular workers are i.i.d. too. This fact motivates two mainstream methods to deal with Byzantine attacks: \textit{attack detection} and \textit{robust aggregation}. For \textit{attack detection}, \cite{Li2019} and \cite{Li2020} propose to offline train an autoencoder, which is used to online calculate credit scores of the workers. The messages sent by the workers with lower credit scores will be discarded in the mean aggregation. The robust subgradient push algorithm in \cite{Ravi2019} operates over a decentralized network. Each worker calculates a score for each of its neighbors, and isolates those with lower scores. The works of \cite{Alistarh2018,Zhu2021} detect the Byzantine workers with historic gradients so as to ensure robustness. The work of \cite{Chen2018} uses redundant gradients for attack detection. However, it requires overlapped data samples on multiple workers, and does not fit for the federated learning setting. For \textit{robust aggregation}, 
the master can use geometric median, instead of mean, to aggregate the received messages \cite{Chen2018b, Xie2018a, Cao2020}.
When the number of Byzantine workers is less than the number of regular workers, geometric median provides a reliable approximation to the average of regular workers' stochastic gradients. Other similar robust aggregation rules include marginal trimmed mean and dimensional median \cite{Yin2018a, Xie2018b, Xie2018c}. Some aggregation rules select a representative stochastic gradient from all the received ones to update the global variable, e.g., Medoid \cite{Xie2018a} and Krum \cite{Blanchard2017}. Medoid selects the stochastic gradient with the smallest distance from all the others, while Krum selects the one with the smallest squared distance to a fixed number of nearest stochastic gradients. An extension of Krum, termed as $h$-Krum, selects $h$ stochastic gradients with Krum and uses their average. Bulyan \cite{Mhamdi2018} first selects a number of stochastic gradients with Krum or other robust selection/aggregation rules, and then uses their trimmed dimensional median.

When the training data and the stochastic gradients are non-i.i.d. at the workers, which is common in federated learning applications \cite{Kevin2019}, naive robust aggregation of stochastic gradients no longer works. The works of \cite{He2020-rs,Peng2020-rs} adopt a resampling strategy to alleviate the effect caused by non-i.i.d. training data. With a larger resampling parameter, the algorithms can handle higher data heterogeneity, at the cost of tolerating less Byzantine workers. Robust stochastic aggregation (RSA) aggregates local variables, instead of stochastic gradients \cite{Liping2018}. To be specific, it considers a total variation (TV) norm-penalized approximation formulation to handle Byzantine attacks, where the TV norm penalty forces the regular workers' local variables to be close, and meanwhile, tolerates the outliers sent by the Byzantine workers. Although the stochastic subgradient method proposed in \cite{Liping2018} is able to solve the TV norm-penalized approximation formulation, it ignores the separable problem structure.

Other related works include \cite{Wu2019,Mahdi2020,Khanduri2019, Karimireddy2021}, which shows that the stochastic gradient noise affects the effectiveness of robust aggregation rules. Thus, the robustness of the Byzantine-resilience methods can be improved by reducing the variance of stochastic gradients. The asynchronous Byzantine-robust SGD is considered in \cite{Damaskinos2018,Yang2020,Xie2019}. The work of \cite{Yin2018b} addresses the saddle-point attacks in the non-convex setting, and \cite{Yang2019a,Yang2019b,Guo2020,Peng2021-dec} consider Byzantine robustness in decentralized learning.

\textbf{Our contributions}. Our contributions are three-fold.

 (i) We propose a Byzantine-robust stochastic alternating direction method of multipliers (ADMM) that utilizes the separable problem structure of the TV norm-penalized approximation formulation. The stochastic ADMM updates are further simplified, such that the iteration-wise communication and computation costs are the same as those of the stochastic subgradient method.

 (ii) We theoretically prove that the proposed stochastic ADMM converges to a bounded neighborhood of the optimal solution at a rate of $O(1/k)$ under mild assumptions, where $k$ is the number of iterations and the size of neighborhood is determined by the number of Byzantine workers.


 (iii) We conduct numerical experiments on the MNIST and COVERTYPE datasets to demonstrate the effectiveness of the proposed stochastic ADMM to various Byzantine attacks.

\section{Problem Formulation}
\label{sec:2}

Let us consider a distributed network with a master and $m$ workers, among which $q$ workers are Byzantine and the other $r = m-q$ workers are regular. The exact value of $q$ and the identities of the Byzantine workers are all unknown. We are interested in solving a stochastic optimization problem in the form of
\vspace{-0.5em}
\begin{align}\label{eq1}
\min\limits_{\tilde{x} } ~ \sum^{m}_{i=1}\mathbb{E}[F(\tilde{x},\xi_i)]+f_0(\tilde{x}),
\end{align}
where $\tilde{x} \in \mathbb{R}^d$ is the optimization variable, $f_0(\tilde{x})$ is the regularization term known to the master, and $F(\tilde{x}, \xi_i)$ is the loss function of worker $i$ with respect to a random variable $\xi_i \sim \mathcal{D}_i$. Here we assume that the data distributions $\mathcal{D}_i$ on the workers can be different, which is common in federated learning applications.

Define $\mathcal{R}$ and $\mathcal{B}$ as the sets of regular workers and Byzantine workers, respectively. We have $|\mathcal{B}| = q$ and $|\mathcal{R}| = r$. Because of the existence of Byzantine workers, directly solving \eqref{eq1} without distinguishing between regular and Byzantine workers is meaningless. A less ambitious alternative is to minimize the summation of the regular workers' local expected cost functions plus the regularization term, in the form of
\begin{align}\label{eq2}
\min\limits_{\tilde{x} } ~ \sum_{i \in \mathcal{R}}\mathbb{E}[F(\tilde{x},\xi_i)]+f_0(\tilde{x}).
\end{align}

Our proposed algorithm and RSA \cite{Liping2018} both aggregate optimization variables, instead of stochastic gradients. To do so, denote $x_i$ as the local copy of $\tilde{x}$ at a regular worker $i \in \mathcal{R}$, and $x_0$ as the local copy at the master. Collecting the local copies in a vector $x = [x_0; \cdots; x_i; \cdots] \in \mathbb{R}^{(r+1)d}$, we know that \eqref{eq2} is equivalent to
\begin{align}\label{eq3}
\min\limits_{x} & ~ \sum_{i\in\mathcal{R}}\mathbb{E}[F(x_i,\xi_i)]+f_0(x_0), \\
s.t.            & ~ x_i - x_0 = 0, ~ \forall i \in \mathcal{R}, \nonumber
\end{align}
where $x_i - x_0 = 0$, $\forall i \in \mathcal{R}$ are the consensus constraints to force the local copies to be the same.

RSA \cite{Liping2018} considers a TV norm-penalized approximation formulation of \eqref{eq3}, in the form of
\begin{align}\label{eq4}
\min\limits_{x }\sum_{i\in\mathcal{R}}(\mathbb{E}[F(x_i,\xi_i)] + \lambda\Vert x_i-x_0\Vert_1)+f_0(x_0),
\end{align}
where $\lambda$ is a positive constant and $\sum_{i\in\mathcal{R}} \Vert x_i-x_0\Vert_1$ is the TV norm penalty for the constraints in \eqref{eq3}. The TV norm penalty forces the regular workers' local optimization variables to be close to the master's, and meanwhile, tolerates the outliers when the Byzantine attackers are present. Due to the existence of the nonsmooth TV norm term, RSA solves \eqref{eq4} with the stochastic subgradient method. The updates of RSA, at the existence of Byzantine workers, are as follows. At time $k$, the master sends $x_0^k$ to the workers, every regular worker $i \in \mathcal{R}$ sends $x_i^k$ to the master, while every Byzantine worker $j \in \mathcal{B}$ sends an arbitrary malicious vector $u_j^k \in \mathbb{R}^d$ to the master.
Then, the updates of $x^{k+1}_i$ for every regular worker $i$ and $x^{k+1}_0$ for the master are given by
\begin{align}
x^{k+1}_i  =&~ x_i^k -\alpha^k \left( F'(x_i^k,\xi^k_i) + \lambda sgn(x_i^k-x_0^k) \right), \nonumber \\
x^{k+1}_0  =&~ x_0^k -\alpha^k \Big( f_0'(x_0^k) - \lambda \sum_{i \in \mathcal{R}} sgn(x_i^k-x_0^k) - \lambda \sum_{j \in \mathcal{B}} sgn(u_j^k-x_0^k) \Big),\label{eq:rsa-update}
\end{align}
where $F'(x_i^k,\xi^k_i)$ is a stochastic gradient at $x_i^k$ respect to a random sample $\xi_i^k$ for regular worker $i$, $sgn(\cdot)$ is the element-wise sign function ($sgn(a)=1$ if $a>0$, $sgn(a)=-1$ if $a<0$, and $sgn(a)\in[-1,1]$ if $a=0$), and $\alpha^k$ is the diminishing learning rate at time $k$.

Although RSA has been proven as a robust algorithm under Byzantine attacks \cite{Liping2018}, the sign functions therein enable the Byzantine workers to send slightly modified messages that remarkably biases the learning process. In addition, RSA fully ignores the special separable structure of the TV norm penalty. In this paper, we also consider the TV norm-penalized approximation formulation \eqref{eq4}, propose a stochastic ADMM that utilizes the problem structure, and develop a novel Byzantine-robust algorithm.


\section{Algorithm Development}
\label{sec:3}
In this section, we utilize the separable problem structure of \eqref{eq4} and propose a robust stochastic ADMM to solve it. The challenge is that the unknown Byzantine workers can send faulty messages during the optimization process. At this stage, we simply ignore the existence of Byzantine workers and develop an algorithm to solve \eqref{eq4}. Then, we will consider the influence of Byzantine workers on the algorithm. We begin with applying the stochastic ADMM to solve \eqref{eq4}, and then simplify the updates such that the iteration-wise communication and computation costs are the same as those of the stochastic subgradient method in \cite{Liping2018}.

\textbf{Stochastic ADMM}. Suppose that all the workers are regular such that $m=r$. To apply the stochastic ADMM, for every worker $i$, introduce auxiliary variables $z(0,i), z(i,0) \in \mathbb{R}^d$ on the directed edges $(0,i), (i,0)$, respectively. By introducing consensus constraints $z(i,0)=x_0$ and $z(0,i)=x_i$, \eqref{eq4} is equivalent to
\begin{align}\label{eq5}
\hspace{-1em} \min\limits_{x, z} & ~ \sum_{i\in\mathcal{R}} (\mathbb{E}[F(x_i,\xi_i)]+ \lambda\Vert z(0,i)-z(i,0)\Vert_1)+f_0(x_0), \\
              s.t. & ~ z(i,0)-x_0=0, ~ z(0,i)-x_i=0, ~ \forall i \in \mathcal{R}. \nonumber
\end{align}
For the ease of presentation, we stack these auxiliary variables in a new variable $z \in \mathbb{R}^{2rd}$. As we will see below, the introduction of $z$ is to split the expectation term $\sum_{i\in\mathcal{R}} \mathbb{E}[F(x_i,\xi_i)]$ and the TV norm penalty term $\sum_{i\in\mathcal{R}} \Vert x_i-x_0\Vert_1$ so as to utilize the separable problem structure.

The augmented Lagrangian function of \eqref{eq5} is
\begin{align} \label{eq6}
\mathcal{L}_{\beta}(x,z,\eta) = &\sum_{i \in \mathcal{R}}(\mathbb{E}[F(x_i,\xi_i)]+ \lambda\Vert z(0,i)-z(i,0)\Vert_1)+f_0(x_0) \nonumber \\
+&\sum_{i \in \mathcal{R}} \left(\langle \eta(i,0),z(i,0)-x_0 \rangle +\frac{\beta}{2}\Vert z(i,0)-x_0\Vert^2 \right)
+\sum_{i \in \mathcal{R}} \left(\langle \eta(0,i),z(0,i)-x_i \rangle +\frac{\beta}{2}\Vert z(0,i)-x_i\Vert^2 \right),
\end{align}
where $\beta$ is a positive constant, while $\eta(i,0) \in \mathbb{R}^d$ and $\eta(0,i) \in \mathbb{R}^d$ are the Lagrange multipliers attached to the consensus constraints $z(i,0)-x_0 = 0$ and $z(0,i)-x_i = 0$, respectively. For convenience, we also collect all the Lagrange multipliers in a new variable $\eta \in \mathbb{R}^{2rd}$.

Given the augmented Lagrangian function \eqref{eq6}, the vanilla ADMM works as follows.
At time $k$, it first updates $x^{k+1}$ through minimizing the augmented Lagrangian function at $z=z^k$ and $\eta=\eta^k$, then updates $z^{k+1}$ through minimizing the Lagrangian function at $x=x^{k+1}$ and $\eta=\eta^k$, and finally updates $\eta^{k+1}$ through dual gradient ascent.
The updates are given by
\begin{subequations} \label{eq8}
\begin{align}
x^{k+1} & = \arg\min\limits_{x } \mathcal{L}_\beta (x,z^k,\eta^k), \label{eq8a}\\
z^{k+1} & = \arg\min\limits_{z } \mathcal{L}_\beta (x^{k+1},z,\eta^k), \label{eq8b}\\
\eta^{k+1}(i,0) & = \eta^{k}(i,0) + \beta(z^{k+1}(i,0)-x_0^{k+1}),\quad
\eta^{k+1}(0,i) = \eta^{k}(0,i) + \beta(z^{k+1}(0,i)-x_i^{k+1}). \label{eq8c}
\end{align}
\end{subequations}

However, the $x$-update in \eqref{eq8a} is an expectation minimization problem and hence nontrivial. To address this issue, \cite{Hua2013} proposes to replace the augmented Lagrangian function $\mathcal{L}_\beta (x,z^k,\eta^k)$ with its stochastic counterpart, given by
\begin{align} \label{eq7}
\mathcal{L}_{\beta}^k(x,z,\eta) = & \sum_{i \in \mathcal{R}} \lambda\Vert z(0,i)-z(i,0)\Vert_1
+f_0(x_0^k) + \langle f'_0(x_0^k), x_0 \rangle + \frac{\sigma^k \Vert x_0-x_0^k\Vert^2}{2} \nonumber \\
+& \sum_{i \in \mathcal{R}} \left (F(x_i^k,\xi_i^k) + \langle F'(x_i^k,\xi_i^k), x_i \rangle + \frac{\sigma^k\Vert x_i-x_i^k\Vert^2}{2} \right) \nonumber \\
+&\sum_{i \in \mathcal{R}} \left(\langle \eta(i,0),z(i,0)-x_0 \rangle +\frac{\beta}{2}\Vert z(i,0)-x_0\Vert^2 \right)
+\sum_{i \in \mathcal{R}} \left(\langle \eta(0,i),z(0,i)-x_i \rangle +\frac{\beta}{2}\Vert z(0,i)-x_i\Vert^2 \right),
\end{align}
where $\xi_i^k$ is the random variable of worker $i$ at time $k$ and $\sigma^k$ is the positive stepsize. Observe that \eqref{eq7} is a stochastic first-order approximation to \eqref{eq6}, in the sense that
\begin{align}
\mathbb{E}[F(x_i,\xi_i)] & \simeq F(x_i^k,\xi_i^k) + \langle F'(x_i^k,\xi_i^k), x_i \rangle + \frac{\sigma^k\Vert x_i-x_i^k\Vert^2}{2} \quad \text{and} \quad
f_0(x) \simeq f_0(x_0^k) + \langle f'_0(x_0^k), x_0 \rangle + \frac{\sigma^k\Vert x_0-x_0^k\Vert^2}{2}, \nonumber
\end{align}
at the points $x_i = x_i^k$ and $x_0 = x_0^k$, respectively.

With the stochastic approximation, the explicit solutions of $x_i^{k+1}$ and $x_0^{k+1}$ are
\begin{align}
x^{k+1}_i & =  \frac{1}{\sigma^k+\beta} \left( \sigma^k x_i^k + \beta z^k(0,i) + \eta^k(0,i) - F'(x_i^k,\xi^k_i)  \right), \nonumber \\
x^{k+1}_0 & =  \frac{1}{\sigma^k+m\beta} \Big( \sigma^k x_0^k  + \sum_{i \in \mathcal{R}}( \beta z^k(i,0)+\eta^k(i,0)) - f_0'(x_0^k)  \Big). \label{eq-explicit-xx}
\end{align}
For simplicity, we replace the parameter $\frac{1}{\sigma^k+\beta}$ by $\alpha_i^k$ and $\frac{1}{\sigma^k+\sum_{i \in \mathcal{R}}\beta}$ by $\alpha_0^k$. Thus, \eqref{eq-explicit-xx} is equivalent to
\begin{align}
&x^{k+1}_i  = x_i^k -\alpha_i^k \left( F'(x_i^k,\xi^k_i) + \beta x_i^k - \beta z^k(0,i) - \eta^k(0,i) \right), \nonumber \\
&x^{k+1}_0  = x_0^k -\alpha_0^k \Big( f_0'(x_0^k) +  \sum_{i \in \mathcal{R}} (\beta x_0^k - \beta z^k(i,0)-\eta^k(i,0)) \Big). \label{eqxx}
\end{align}

\textbf{Simplification}. Observe that the $z$-update in \eqref{eq8b} is also challenging as the variables $z(0,i)$ and $z(i,0)$ are coupled by the TV norm penalty term. Next, we will simplify the three-variable updates in \eqref{eqxx}, \eqref{eq8b} and \eqref{eq8c} to eliminate the $z$-update and obtain a more compact algorithm. Note that the decentralized deterministic ADMM can also be simplified to eliminate auxiliary variables \cite{sta-opt}. However, we are considering the distributed stochastic ADMM, and the TV norm penalty term makes the simplification much more challenging.

\begin{proposition}[\textbf{Simplified stochastic ADMM}]\label{lem:simp}
Suppose $m=r$. The updates \eqref{eqxx}, \eqref{eq8b} and \eqref{eq8c} can be simplified as
\begin{align}\label{eqxi-update}
x^{k+1}_i = &~x_i^k -\alpha_i^k \left( F'(x_i^k,\xi^k_i) + 2\eta^k_i - \eta^{k-1}_i  \right),\\
\label{eqx0-update} x^{k+1}_0 =& ~x_0^k -\alpha_0^k \bigg( f_0'(x_0^k) - \sum\limits_{i \in \mathcal{R}} (2\eta^k_i - \eta^{k-1}_i) \bigg),\\
\eta^{k+1}_i:=&~\eta^{k+1}(i,0)=-\eta^{k+1}(0,i)
= \mathrm{proj}_{\lambda}\left(\eta^k_i+\frac{\beta}{2}(x^{k+1}_i-x^{k+1}_0) \right), \label{eq11}
\end{align}
where $\mathrm{proj}_{\lambda}(\cdot)$ is the projection operator that for each dimenison maps any point in $\mathbb{R}$ onto $[-\lambda,\lambda]$.
\end{proposition}
\begin{proof} See \ref{app:lem1}.\end{proof}

\textbf{Presence of Byzantine workers}. Now we start to consider how the stochastic ADMM updates \eqref{eqxi-update}, \eqref{eqx0-update} and \eqref{eq11} are implemented when the Byzantine workers are present. At time $k$, every regular worker $i \in \mathcal{R}$ updates $x_i^{k+1}$ with \eqref{eqxi-update} and $\eta^{k+1}_i$ with \eqref{eq11}, and then sends $\eta_i^{k+1}$ to the master. Meanwhile, every Byzantine worker $j \in \mathcal{B}$ can cheat the master by sending $\eta_j^{k+1} \mathbb{R}^d$ where the elements are arbitrary within $[-\lambda, \lambda]^d$. Otherwise, the Byzantine worker $j$ can be directly detected and eliminated by the master. This amounts to that every Byzantine worker $j \in \mathcal{B}$ follows an update rule similar to \eqref{eq11}, as
\begin{align}
\eta^{k+1}_j=\mathrm{proj}_{\lambda}\left(\eta^k_j+\frac{\beta}{2}(u^{k+1}_j-x^{k+1}_0) \right), \label{eq11c}
\end{align}
where $u^{k+1}_j \in \mathbb{R}^d$ is an arbitrary vector. After receiving all messages $\eta_i^{k+1}$ from the regular workers $i \in \mathcal{R}$ and $\eta_j^{k+1}$ from the Byzantine workers $j \in \mathcal{B}$, the master updates $x_0^{k+1}$ via
\begin{align} \label{eqx0-update-1}
x^{k+1}_0 =& ~x_0^k -\alpha_0^k \Big( f_0'(x_0^k) - \sum\limits_{i \in \mathcal{R}} (2\eta^k_i - \eta^{k-1}_i)
- \sum\limits_{j \in \mathcal{B}}(2\eta^k_j - \eta^{k-1}_j) \Big).
\end{align}

The Byzantine-robust stochastic ADMM for distributed learning is outlined in Algorithm \ref{alg} and illustrated in Figure \ref{pic1}. Observe that the communication and computation costs are the same as those in the stochastic subgradient method in \cite{Liping2018}. The only extra cost is that every worker $i$ must store the dual variable $\eta^{k}_i$.

Comparing the stochastic subgradient updates \eqref{eq:rsa-update} with the stochastic ADMM updates \eqref{eqxi-update}, \eqref{eqx0-update} and \eqref{eq11}, we can observe a primal-dual connection. In the stochastic subgradient method, the workers upload primal variables $x_i^k$, while in the stochastic ADMM, the workers upload dual variables $\eta_i^{k+1}$. The stochastic subgradient method controls the influence of a malicious message by the sign function. No matter what the malicious message is, its modification on each dimension is among $-\lambda$, $\lambda$, and a  value within $[-\lambda, \lambda]$ if the values of the malicious worker and the master are identical. The stochastic ADMM controls the influence of a malicious message by the projection function. The modification of the malicious message on each dimension is within $[-\lambda, \lambda]$.

\begin{algorithm}[htb!]
	\caption{Byzantine-Robust Stochastic ADMM}\label{alg}
	\centerline{\textbf{Master}}
    Initialize $x_0^0$, $\eta^{-1}_i$, and $\eta^{0}_i$.
	\begin{algorithmic}[1]
		\While {not stopped}
        \State Update $x_0^{k+1}$ via \eqref{eqx0-update-1};
        \State Broadcast $x_0^{k+1}$ to all the workers;
        \State Receive $\eta_i^{k+1}$ from the regular workers $i \in \mathcal{R}$ and $\eta_j^{k+1}$ from the Byzantine workers $j \in \mathcal{B}$.
 		\EndWhile
	\end{algorithmic}
	\centerline{\textbf{Regular Worker $i$}}
    Initialize $x^0_i$, $\eta^{-1}_i$, and $\eta^{0}_i$.
	\begin{algorithmic}[1]
		\While {not stopped}
        \State Update $x_i^{k+1}$ via \eqref{eqxi-update};
        \State Update $\eta^{k+1}_i$  via \eqref{eq11};
        \State Send $\eta^{k+1}_i$ to the master;
        \State Receive $x_0^{k+1}$ from the master.
 		\EndWhile
	\end{algorithmic}
\end{algorithm}

\begin{figure}[htb!]
  \centering
    \includegraphics[width=0.5\textwidth]{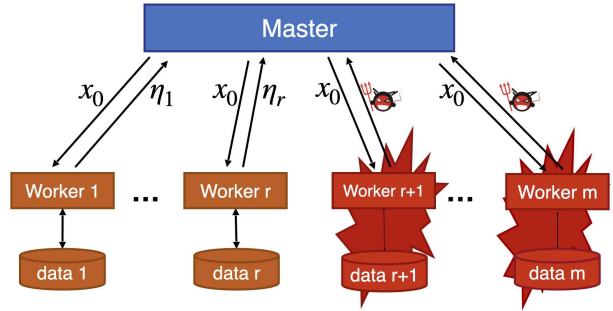}
    \caption{An illustration of the distributed stochastic ADMM for Byzantine-robust learning. In the master-worker architecture, there are $m$ workers in total, among which $r$ are regular workers and $q = m-r$ are Byzantine workers. }
  \label{pic1} 
\end{figure}

\section{Convergence Analysis}
\label{sec:4}

In this section, we analyze the convergence of the proposed Byzantine-robust stochastic ADMM. We make the following assumptions, which are common in analyzing distributed stochastic optimization algorithms.

\begin{assumption}[\textbf{Strong convexity}]\label{ass1}
The local cost functions $\mathbb{E}[F(x_i, \xi_i)]$ and the regularization term $f_0(x_0)$ are strongly convex with constants $\mu_i$ and $\mu_0$, respectively.
\end{assumption}

\begin{assumption} [\textbf{Lipschitz continuous gradients}] \label{ass2}
The local cost functions $\mathbb{E}[F(x_i, \xi_i)]$ and the regularization term $f_0(x_0)$ have Lipschitz continuous gradients with constants $L_i$ and $L_0$, respectively.
\end{assumption}

\begin{assumption}[\textbf{Bounded variance}]\label{ass:var}
Within every worker $i$, the data sampling is i.i.d. with $\xi^k_i \sim \mathcal{D}_i$.  The variance of stochastic gradients $F'({x}, \xi_i)$ is upper bounded by $\delta^2_i$, as
\begin{equation}\mathbb{E} \| F'({x}, \xi_i)- \mathbb{E}[F'({x}, \xi_i)] \|^2 \leq \delta^2_i.\label{eq:ass-var}\end{equation}
\end{assumption}

\subsection{Main Results}

First, we show the equivalence between \eqref{eq2} and \eqref{eq5}. When the penalty parameter $\lambda$ is sufficiently, it has been shown in Theorem 1 of \cite{Liping2018} that the optimal primal variables of \eqref{eq5} are consensual and identical to the minimizer of \eqref{eq2}. We repeat this conclusion in the following lemma.

\begin{lemma}[\textbf{Consensus and equivalence}]\label{Coro:1}
Suppose Assumptions \ref{ass1} and \ref{ass2} hold. If {$\lambda\ge\lambda_0:=\max_{i\in\mathcal{R}}\|\mathbb{E}[F'({\tilde{x}^*}, \xi_i)]\|_\infty$}, then for all $i \in \mathcal{R}$, we have
$$x_i^*=x_0^*=\tilde{x}^*, $$
where $x_i^*$ and $x_0^*$ are the optimal primal variables of \eqref{eq5}, and $\tilde{x}^*$ is the minimizer of \eqref{eq2}.
\end{lemma}

Intuitively, setting a sufficiently large penalty parameter $\lambda$ ensures the variables $x_i$ and $x_0$ to be consensual, since a larger $\lambda$ gives more weight on the consensus constraints. When the training data at the workers are non-i.i.d., the local expected gradients $\mathbb{E}[F'({\tilde{x}^*}, \xi_i)]$ deviate from 0, which leads to a large lower bound $\lambda_0$ to maintain consensus. Once the variables are consensual, \eqref{eq5} is equivalent to \eqref{eq2}.

Now, we present the main theorem on the convergence of the proposed Byzantine-robust stochastic ADMM.
\begin{theorem}[\textbf{$O(1/k)$-convergence}] \label{thm:conv}
Suppose Assumptions \ref{ass1}, \ref{ass2}, and \ref{ass:var} hold. Let $\lambda\ge\lambda_0$
and the stepsizes be
$$\alpha_0^k=\min\left\{\frac1{ck+m\beta}, \frac1{\mu_0+L_0}, \frac1{\mu_i+L_i}\right\}, \quad
\alpha_i^k=\min\left\{\frac1{ck+\beta}, \frac1{\mu_0+L_0}, \frac1{\mu_i+L_i}\right\},\quad \forall i\in\mathcal{R},$$ for some positive constants $c<\min\left\{\frac{\mu_0L_0}{\mu_0+L_0}, \frac{\mu_iL_i}{\mu_i+L_i}: i \in \mathcal{R}\right\}$ and $\beta>0$.
Then we have
\begin{equation}
\mathbb{E} \|x_0^k-x_0^*\|^2+\sum_{i \in \mathcal{R}}\mathbb{E} \|x_i^k-x_i^*\|^2=O(1/k)+O(\lambda^2q^2).
\label{eq:Vk-bound}
\end{equation}
\end{theorem}

\begin{proof}
See \ref{app:conv}.
\end{proof}

Theorem \ref{thm:conv} guarantees that if we choose the stepsizes for both the workers and the master in the order of $O(1/k)$, then the Byzantine-robust stochastic ADMM asymptotically approaches to the $O(\lambda^2q^2)$ neighborhood of the optimal solution $\tilde{x}^*$ of \eqref{eq2} (which equals $x_0^*$ and $x_i^*$, according to Lemma \ref{Coro:1}) in an $O(1/k)$ rate. Note that the $O(1/k)$ stepsizes are sensitive to their initial values \cite{Liping2018}. Therefore, we set the $O(1/\sqrt{k})$ stepsizes in the numerical experiments. We also provide in \ref{app:sqk} an ergodic convergence rate of $O(\log k/\sqrt{k})$ with $O(1/\sqrt{k})$ stepsizes.

In \eqref{eq:Vk-bound}, the asymptotic learning error is in the order of $O(\lambda^2q^2)$, which is the same as that of RSA \cite{Liping2018}. When more Byzantine workers are present, $q$ is larger and the asymptotic learning error increases. Using a larger $\lambda$ helps consensus as indicated in Lemma \ref{Coro:1}, but incurs higher asymptotic learning error. In the numerical experiments, we will imperially demonstrate the influence of $q$ and $\lambda$.

\subsection{Comparison with RSA: Case Studies}

The proposed Byzantine-robust stochastic ADMM and RSA \cite{Liping2018} solve the same problem, while the former takes advantages of the separable problem structure. Below we briefly discuss the robustness of the two algorithms to different Byzantine attacks.

RSA is relatively sensitive to small perturbations. To perturb the update of $x_0^{k+1}$ in \eqref{eq:rsa-update}, Byzantine worker $j$ can generate malicious $u_j^k$ that is very close to $x_0^k$, but its influence on each dimension is still $\lambda$ or $-\lambda$. Potentially, this attack is able to lead the update to move toward a given wrong direction. In contrast, for the Byzantine-robust stochastic ADMM, small perturbations on $\eta_j^k$ change little on the update of $x_0^{k+1}$ in \eqref{eqx0-update-1}. To effectively attack the Byzantine-robust stochastic ADMM, Byzantine worker $j$ can set each element of $\eta_j^k$ to be $\lambda (-1)^k$, then its influence on each dimension will oscillate between $3\lambda$ and $-3\lambda$. In comparison, the influence of this attack for RSA is just $\lambda$ or $-\lambda$ on each dimension. However, these large oscillations are easy to distinguish by the master through screening the received messages. In addition, it is nontrivial for this attack to lead the update to move toward a given wrong direction.

Developing Byzantine attacks that are most harmful to the Byzantine-robust stochastic ADMM and RSA, respectively, is beyond the scope of this paper. Instead, we give a toy example and develops two Byzantine attacks to justify the discussions above.

\begin{example}\label{ex}
Consider a one-dimensional distributed machine learning task with $r=2$ regular workers (numbered by $1$ and $2$) and $q=1$ Byzantine worker (numbered by $3$). The functions are deterministic and quadratic, with $f_0(x_0)=x_0^2/2$, $F_1(x_1,\xi_1)=(x_1-1)^2/4$, and $F_2(x_2,\xi_2)=(x_2-1)^2/4$. Therefore, $\tilde{x}^*=1/2$ is the minimizer of \eqref{eq2} and $\lambda_0=1/4$ by Lemma \ref{Coro:1}. The local primal variables are initialized as their local optima, i.e., $x_0^0=0$ and $x_1^0=x_2^0=1$ for both algorithms. The local dual variables of the Byzantine-robust stochastic ADMM are initialized as $\eta^{-1}_i=\eta_i^0=0$ for $i\in\{1,2,3\}$. We construct two simple attacks.

\noindent \textbf{Small value attack.} Byzantine worker $3$ generates $u_3^k = x_0^k - \frac\epsilon{\max\{k(k+1),1\}}$, where $\epsilon>0$ is a perturbation parameter.

\noindent \textbf{Large value attack.} Byzantine worker $3$ generates $u_3^k = x_0^k - \frac{4\lambda}{\beta}(-1)^k$.

\end{example}

%

We choose the parameters $\lambda=\lambda_0=1/2$ and $\beta=1$, with stepsizes $\alpha_0^k= \frac1{k/8+3}$ and $\alpha_1^k=\alpha_2^k=\frac1{k/8+1}$. The perturbation parameter is set as $\epsilon=1/2$ for the small value attack. Figure \ref{fig_ex} shows the values of the local primal variables on the master and the regular workers. For both algorithms and both attacks, the master and the regular workers are able to asymptotically reach consensus as asserted by Lemma \ref{Coro:1}. 
Under the small value attack, RSA has larger asymptotic learning error than the Byzantine-robust stochastic ADMM as we have discussed, while under the large value attack, both  algorithms coincidentally have zero asymptotic learning errors. In addition, we can observe that the Byzantine-robust stochastic ADMM is more stable than RSA under both attacks.

%

\begin{figure}
\centering
\subfigure[Small value attack]{
\label{fig_exa}
\includegraphics[width=0.45\textwidth]{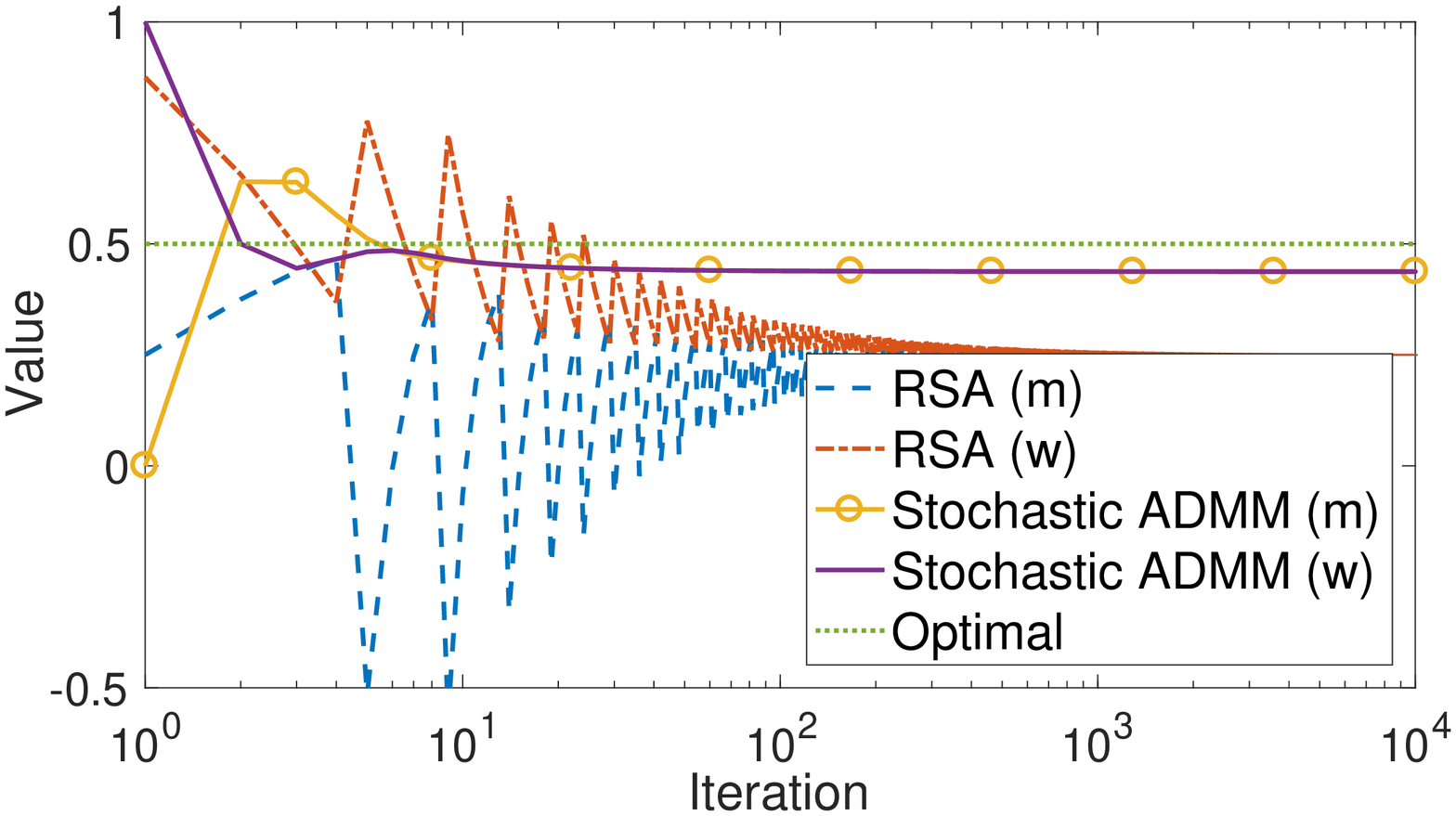}}
\subfigure[Large value attack]{
\label{fig_exb}
\includegraphics[width=0.45\textwidth]{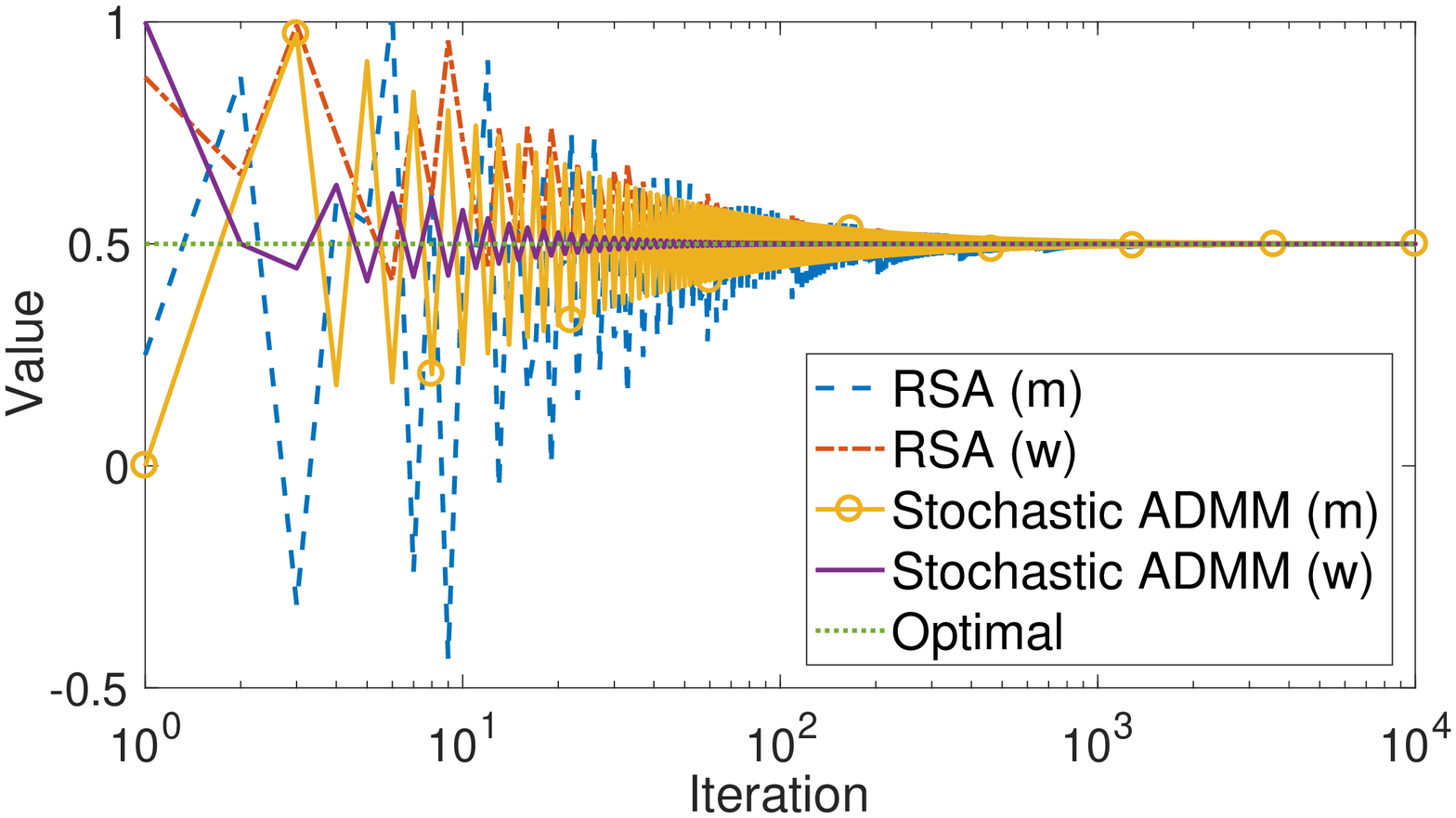}}
\caption{Values of the local primal variables on the master `(m)' and the regular workers `(w)'.}
\label{fig_ex}
\end{figure}

\section{Numerical Experiments}
\label{sec:5}

In this section, we evaluate the robustness of the proposed algorithm to various Byzantine attacks. We compare the proposed Byzantine-robust \textbf{Stochastic ADMM} with the following benchmark algorithms: (i) $\textbf{Ideal SGD}$ without Byzantine attacks; (ii) \textbf{SGD} subject to Byzantine attacks; (iii) \textbf{Geometric median} stochastic gradient aggregation \cite{Chen2018b,Xie2018a}; (iv) \textbf{Median} stochastic gradient aggregation \cite{Chen2018b,Xie2018a}; (v) \textbf{RSA} \cite{Liping2018}. All the parameters of the benchmark algorithms are hand-tuned to the best. Although the stochastic ADMM and RSA are rooted in the same problem formulation \eqref{eq4}, they perform differently for the same value of $\lambda$ due to Byzantine attacks, as we have observed in Example \ref{ex}. Therefore, we hand-tune the best $\lambda$ for the stochastic ADMM and RSA, respectively. In the numerical experiments, we use two datasets, MNIST\footnote{\url{http://yann.lecun.com/exdb/mnist}} and COVERTYPE\footnote{\url{https://www.csie.ntu.edu.tw/~cjlin/libsvmtools/datasets}}. The statistics of these datasets are shown in Table \ref{table1}. We launch one master and 20 workers. In the i.i.d case, we conduct experiments on both datasets by randomly and evenly splitting the data samples to the workers, while in the non-i.i.d. case we only use the MNIST dataset. Each regular worker uses a mini-batch of 32 samples to estimate the local gradient at each iteration. The loss functions $f_i(\tilde{x})$ of workers are softmax regressions, and the regularization term is given by $f_0(\tilde{x}) = \frac{0.01}{2} \Vert \tilde{x} \Vert^2$. Performance is evaluated by the top-1 accuracy.

\begin{table}\centering
	\begin{tabular}{c|c|c|c}
		\hline
		Name  &  Training Samples & Testing Samples & Attributes \\
		\hline
		COVERTYPE & 465264 &  115748 & 54  \\
		\hline
		MNIST & 60000 & 10000 & 784  \\
		\hline
	\end{tabular}
	\caption{Specifications of the datasets.}\label{table1}
\end{table}

\textbf{Gaussian attack}. Under Gaussian attack, at every iteration every Byzantine worker sends to the master a random vector, whose elements follow the Gaussian distribution with standard deviation $100$. Here we set the number of Byzantine workers $q=8$. For Stochastic ADMM on the MNIST dataset, we set $\lambda$ = 0.5,  $\beta = 0.5$, $\alpha_0^k = \frac{1}{10+10\sqrt{k}}$, and $\alpha_i^k = \frac{1}{0.5+10\sqrt{k}}$. As shown in Figure \ref{fig1a}, SGD fails, Stochastic ADMM, RSA and Geometric median perform very similarly and are close to Ideal SGD, while Median is a little worse than the others. On the COVERTYPE dataset, we set $\lambda$ = 0.5, $\beta = 0.1$, $\alpha_0^k = \frac{1}{2+50\sqrt{k}}$, and $\alpha_i^k = \frac{1}{0.1+50\sqrt{k}}$ for Stochastic ADMM. As shown in Figure \ref{fig1b}, SGD performs the worst, while Stochastic ADMM, Geometric median, and median are close to Ideal SGD. Among all the Byzantine-robust algorithms, Stochastic ADMM has the fastest converge speed.
\begin{figure}[ht!]
\centering
\subfigure[MNIST]{
\label{fig1a}
\includegraphics[width=0.45\textwidth]{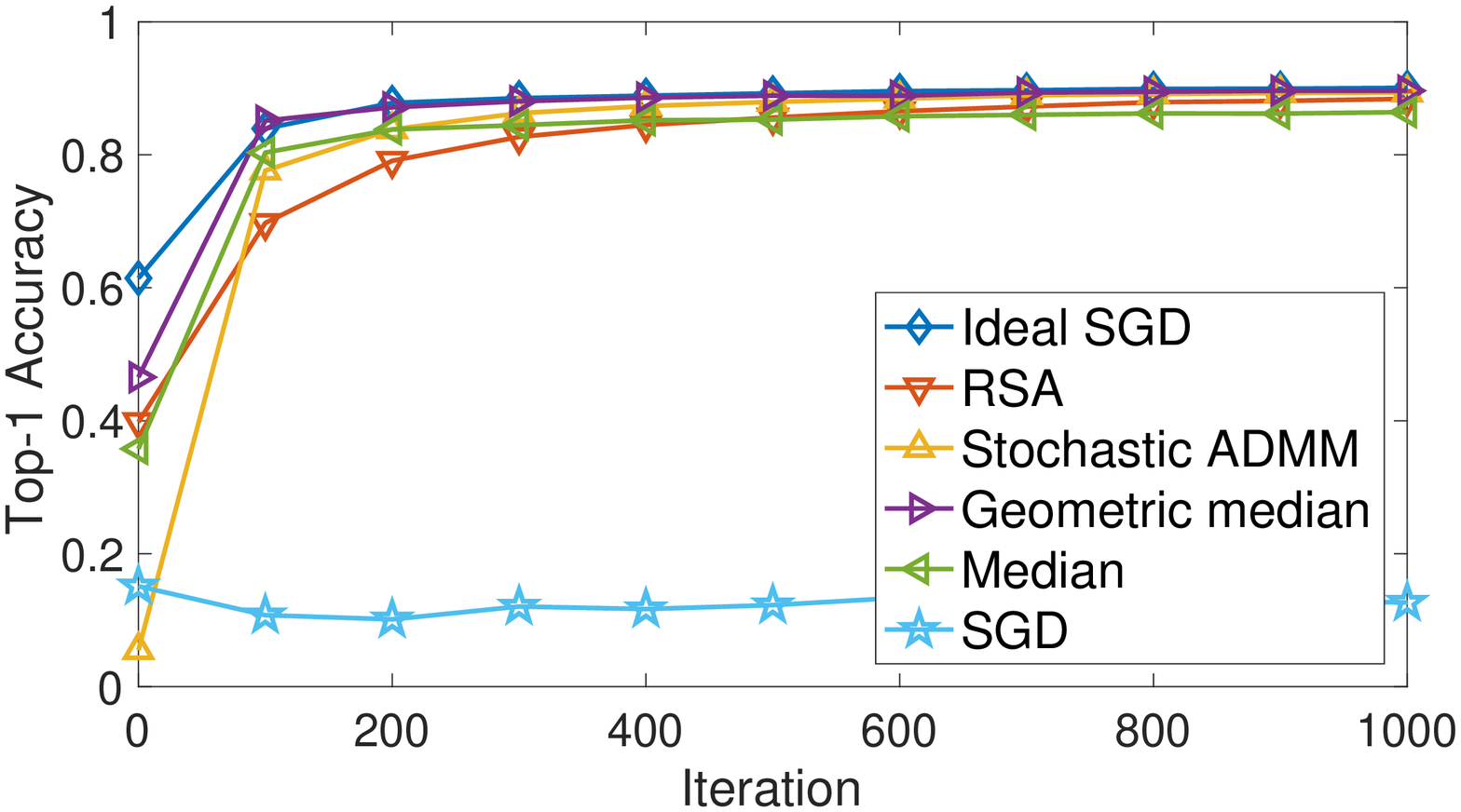}}
\subfigure[COVERTYPE]{
\label{fig1b}
\includegraphics[width=0.45\textwidth]{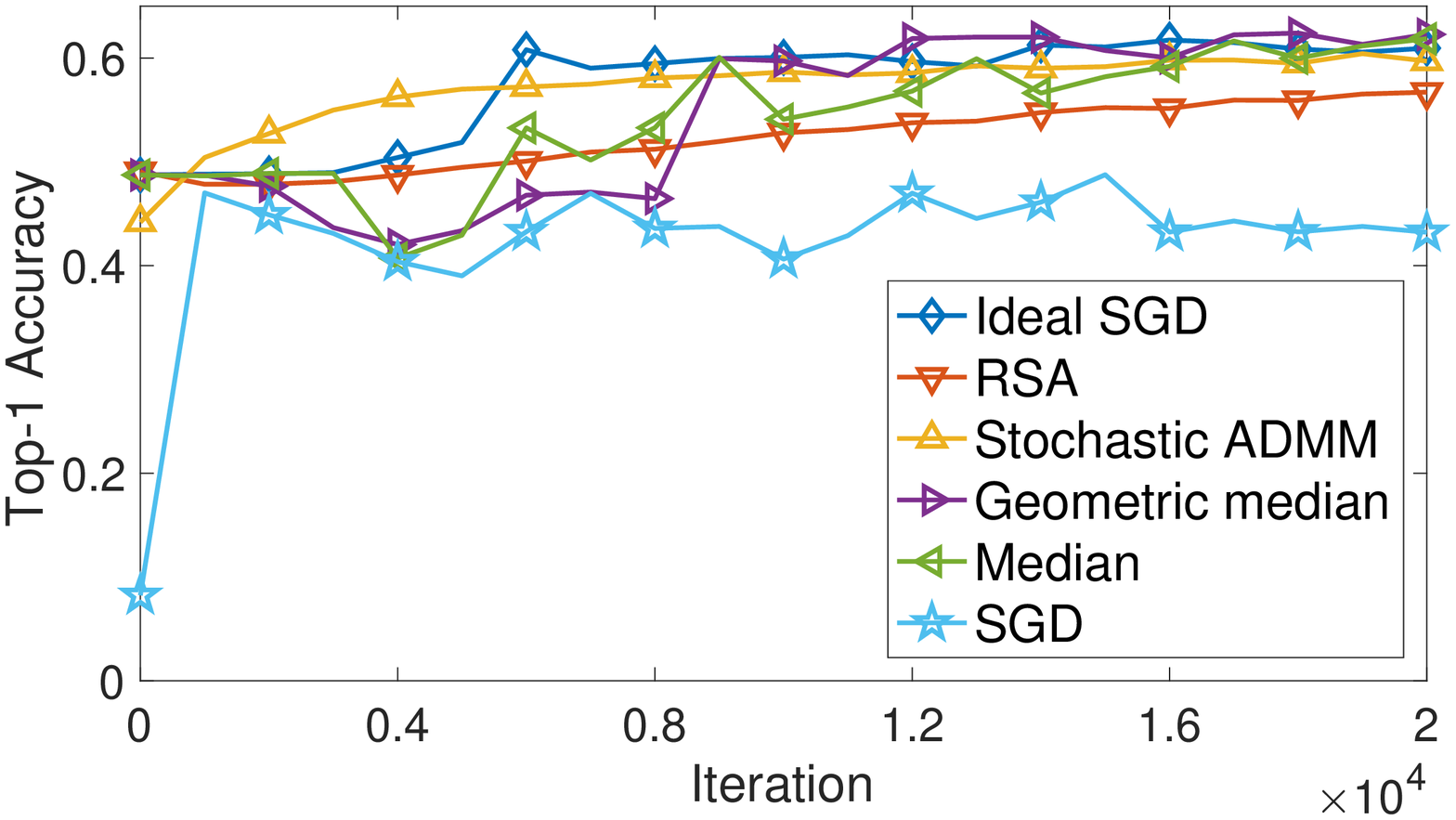}}
\caption{Top-1 accuracy under Gaussian attack when $q = 8$.}
\label{fig1}
\end{figure}

\textbf{Sign-flipping attack}. Under sign-flipping attack, at every iteration every Byzantine worker calculates its local variable, but flips the sign by multiplying with a constant $\varepsilon < 0$, and sends to the master. Here we set $\varepsilon = -3$ and the number of Byzantine workers $q=8$. On the MNIST dataset, we set the parameters of Stochastic ADMM as $\lambda = 0.05$, $\beta = 0.1$, $\alpha_0^k = \frac{1}{2+5\sqrt{k}}$, and $\alpha_i^k = \frac{1}{0.1+5\sqrt{k}}$. Figure \ref{fig2a} shows that SGD also fails in this situation. Stochastic ADMM, RSA, and Geometric median are close to Ideal SGD, and achieve better accuracy than median. In Figure \ref{fig2b}, shows the performance on the COVERTYPE dataset. The parameters of Stochastic ADMM are $\lambda = 0.5$, $\beta = 0.3$, $\alpha_0^k = \frac{1}{6+100\sqrt{k}}$ ,and $\alpha_i^k = \frac{1}{0.3+100\sqrt{k}}$.  Stochastic ADMM and RSA are close to Ideal SGD, while outperform Geometric median and Median.
\begin{figure}[ht!]
\centering
\subfigure[MNIST]{
\label{fig2a}
\includegraphics[width=0.45\textwidth]{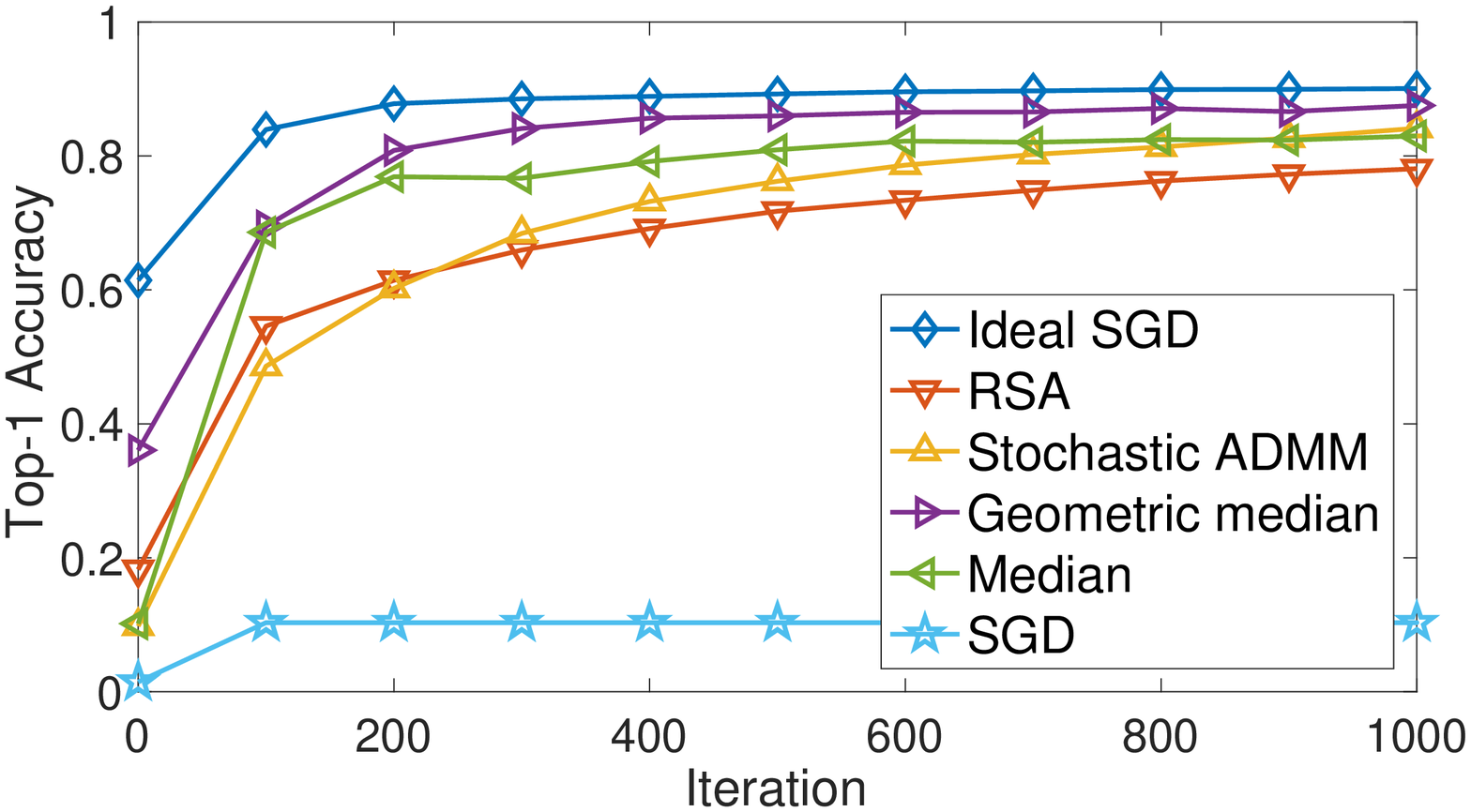}}
\subfigure[COVERTYPE]{
\label{fig2b}
\includegraphics[width=0.45\textwidth]{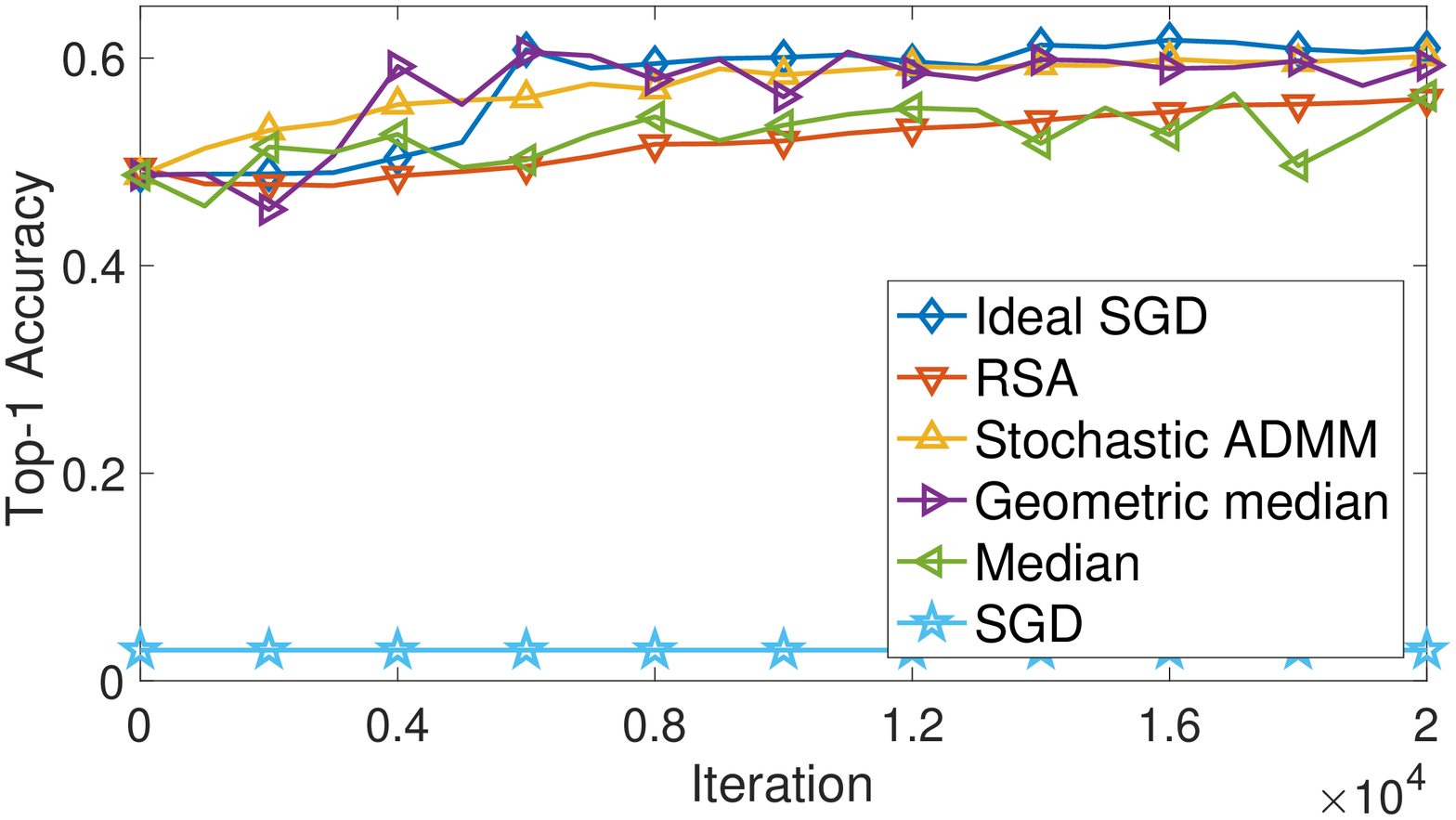}}
\caption{Top-1 accuracy under sign-flipping attack when $q = 8$.}
\label{fig2}
\end{figure}

\textbf{Without Byzantine attack}. We also investigate the case without Byzantine attack in both MNIST and COVERTYPE datasets, as shown in Figure \ref{fig4}. In Figure \ref{fig4a}, Stochastic ADMM on MNIST dataset chooses the parameters $\lambda = 0.5$, $\beta = 0.5$, $\alpha_0 = \frac{1}{10+10\sqrt{k}}$, and $\alpha_i = \frac{1}{0.5+10\sqrt{k}}$. Without Byzantine attack, performance of Stochastic ADMM, RSA, and Geometric median is very similar to Ideal SGD, while Median is worse than the other Byzantine-robust algorithms. On the COVERTYPE dataset, we set the parameters of Stochastic ADMM as $\lambda = 0.5$, $\beta = 0.3$,  $\alpha_0^k = \frac{1}{6+10\sqrt{k}}$, and $\alpha_i^k = \frac{1}{0.3+10\sqrt{k}}$. As shown in Figure \ref{fig4b}, Stochastic ADMM is the best among all the algorithms, and RSA outperforms Geometric median and Median. We conclude that although Stochastic ADMM introduces bias to the updates, it still works well in the attack-free case.
\begin{figure}[ht!]
\centering
\subfigure[MNIST]{
\label{fig4a}
\includegraphics[width=0.45\textwidth]{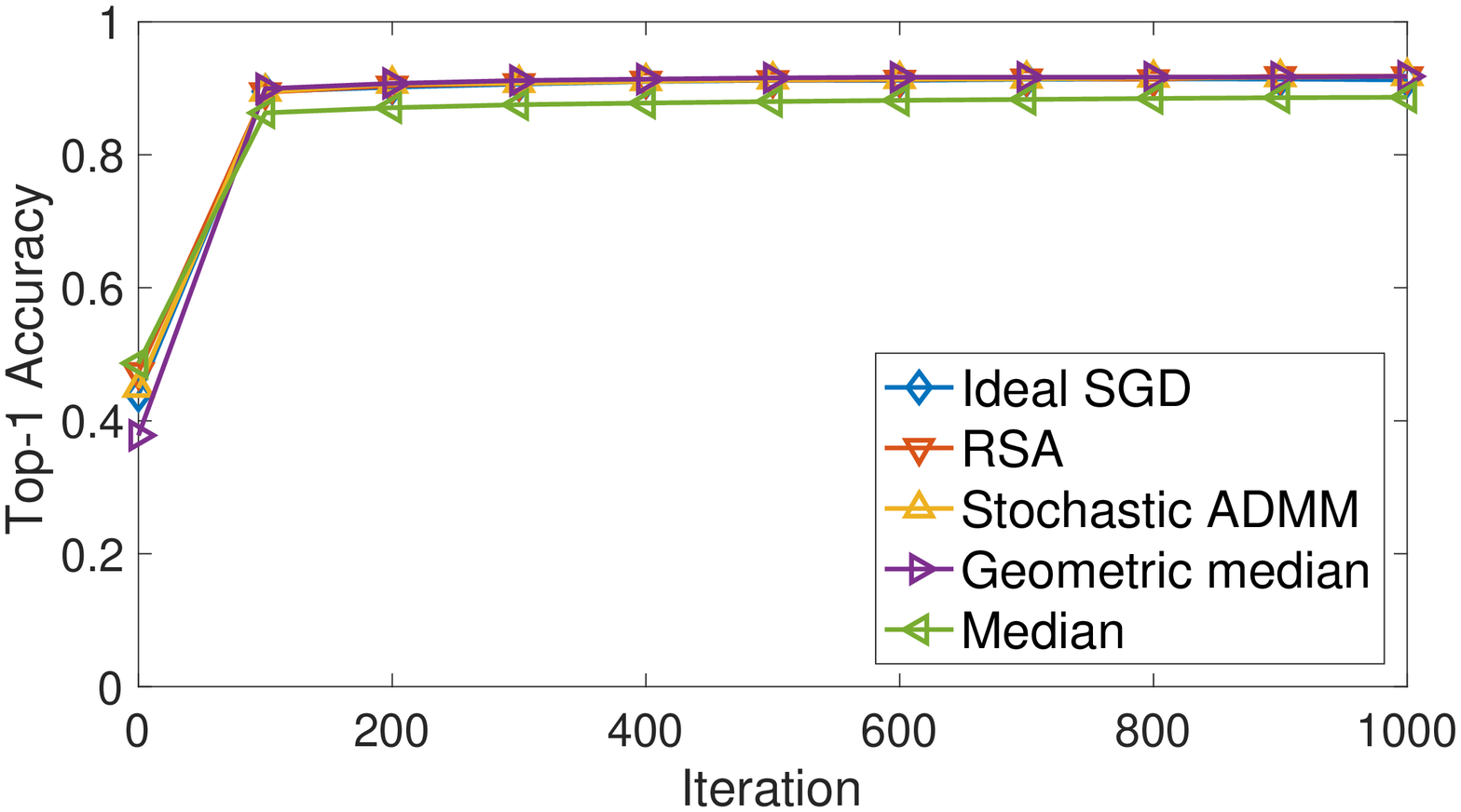}}
\subfigure[COVERTYPE]{
\label{fig4b}
\includegraphics[width=0.45\textwidth]{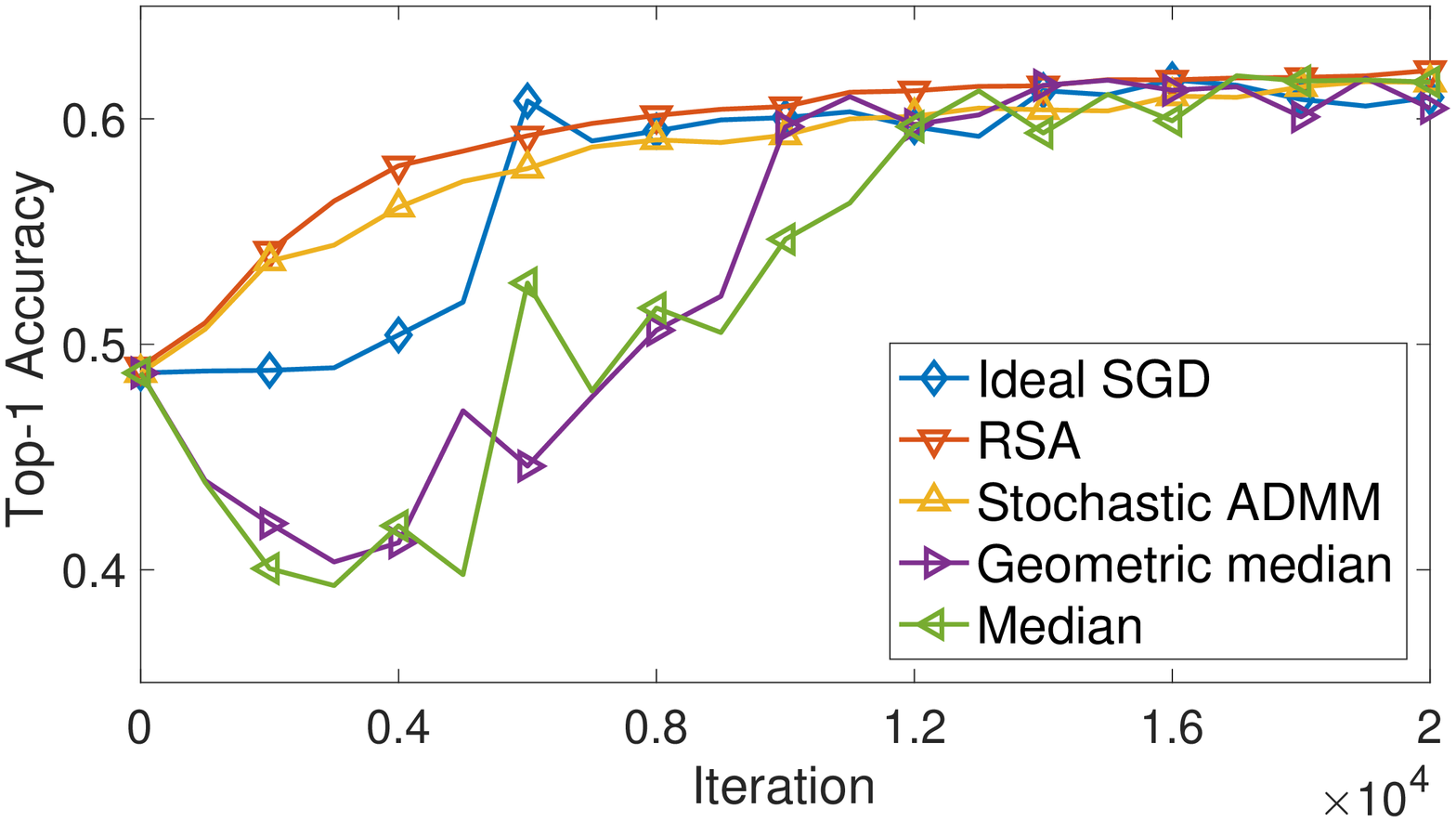}}
\caption{Top-1 accuracy without Byzantine attack.}
\label{fig4}
\end{figure}

\textbf{Impact of $\lambda$.} Here we show how the performance of the proposed algorithm on the two datasets are affected by the choice of the penalty parameter $\lambda$.
We use sign-flipping attack with $\varepsilon = -3$ in the numerical experiments, and the number of Byzantine workers is $q = 4$. The parameters $\beta$, $\alpha_0^k$ and  $\alpha_i^k$ are hand-tuned to the best. As depicted in Figure \ref{fig5}, on both datasets, the performance of Stochastic ADMM degrades when $\lambda$ is too small. The reason is that, when $\lambda$ is too small the regular workers are rely more on their local data, leading to deficiency of the distributed learning system \cite{Liping2018}. Meanwhile, $\lambda$ being too large also leads to worse performance.
\begin{figure}[ht!]
\centering
\subfigure[MNIST]{
\label{fig5a}
\includegraphics[width=0.45\textwidth]{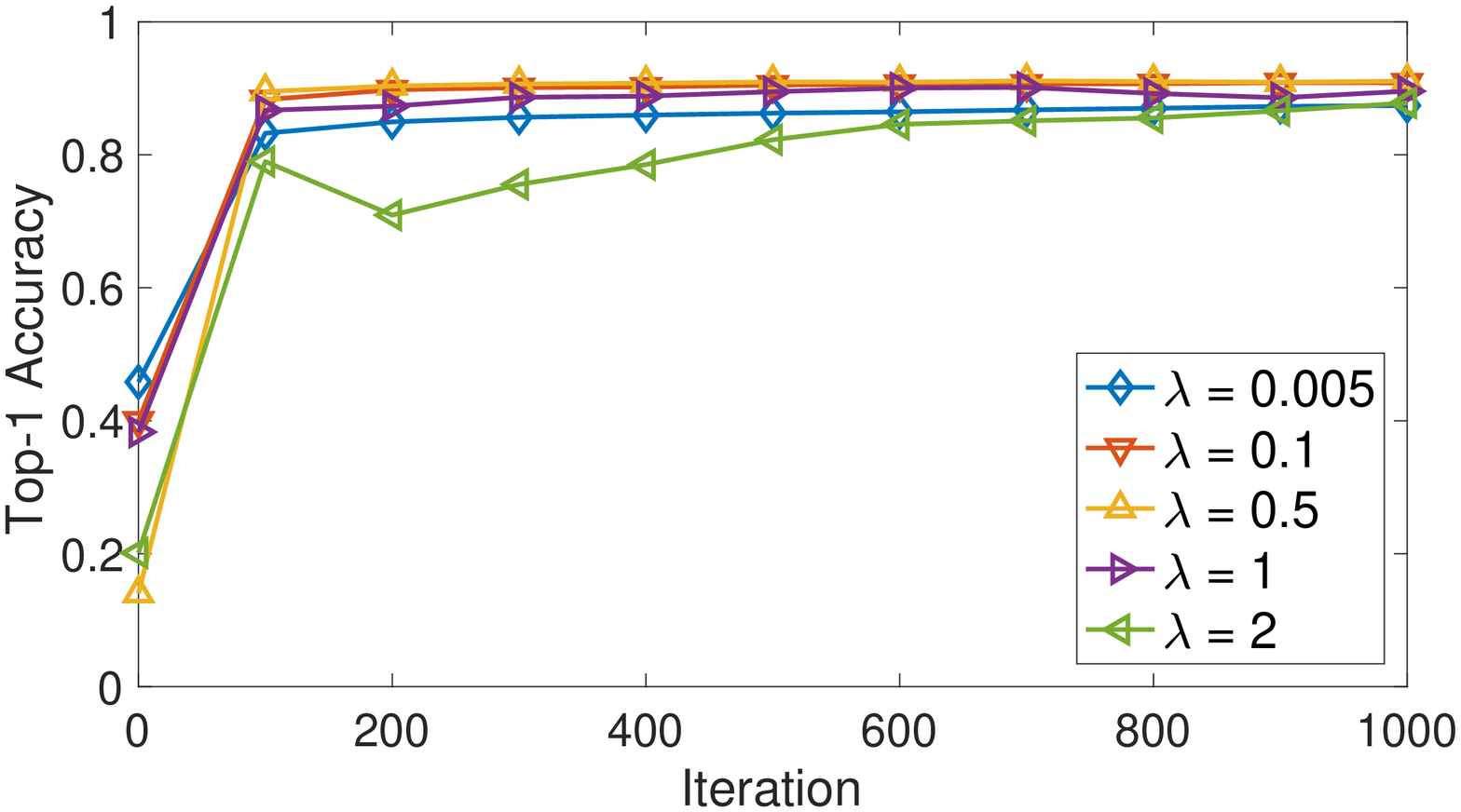}}
\subfigure[COVERTYPE]{
\label{fig5b}
\includegraphics[width=0.45\textwidth]{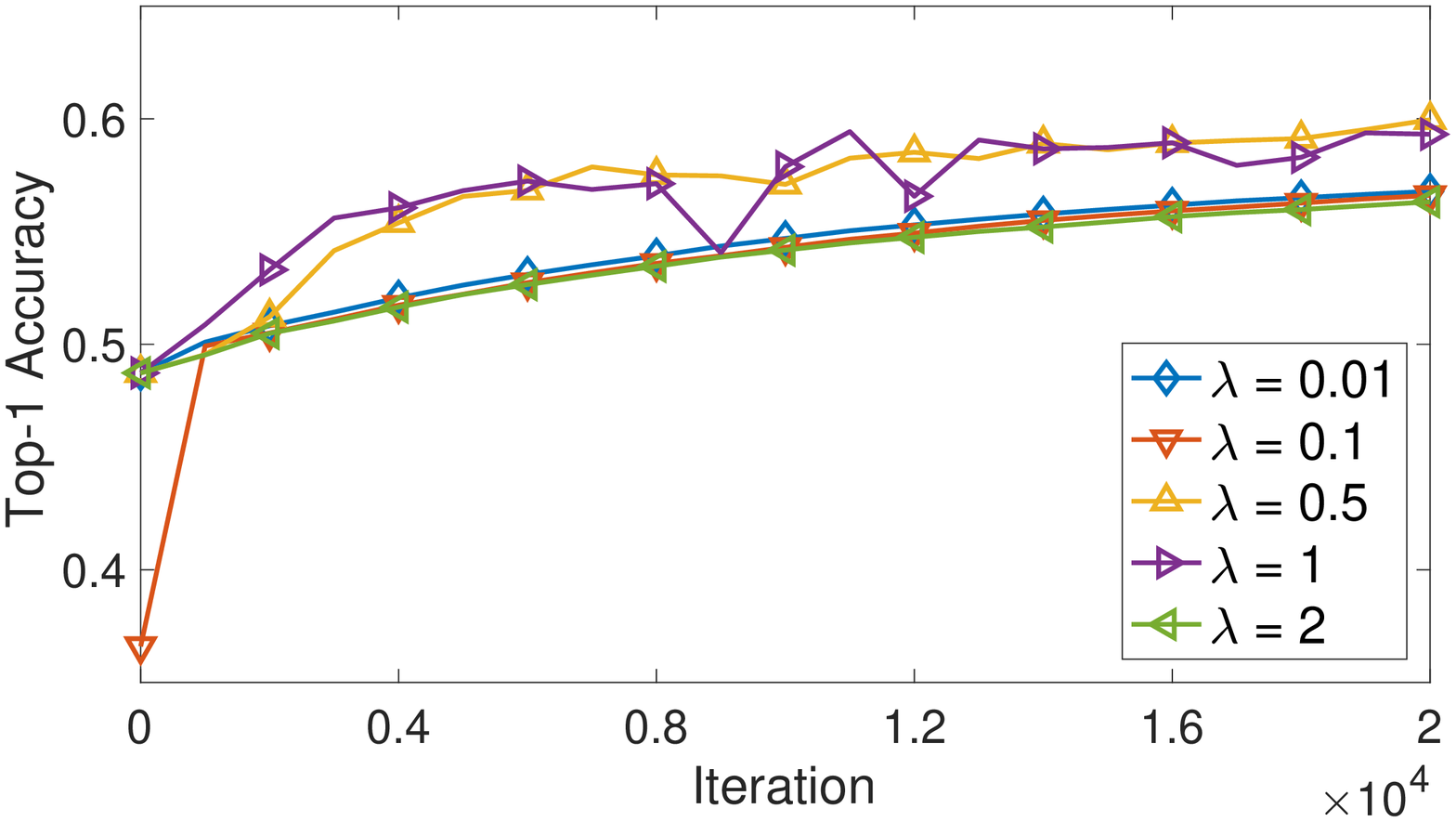}}
\caption{Top-1 accuracy under sign-flipping attacks with different $\lambda$.}
\label{fig5}
\end{figure}

\textbf{Non-i.i.d. data}. To demonstrate the robustness of the proposed algorithm against Byzantine attacks on non-i.i.d. data, we redistribute the MNIST dataset by letting every two workers share one digit. All Byzantine workers $j$ choose one regular worker indexed by $p$, and send $u_j^{k+1} = x_p^{k+1}$ to the master at every iteration $k$. When the number of Byzantine workers is $q = 4$, the best reachable accuracy is around 0.8, because of the absence of two handwriting digits' data. Similarly, when the number of Byzantine workers is $q=8$, the best reachable accuracy is around 0.6. Here, we set the parameters of Stochastic ADMM as $\lambda = 0.8$, $\beta = 0.2$, $\alpha_0^k = \frac{1}{4+10\sqrt{k}}$, and $\alpha_i^k = \frac{1}{0.2+10\sqrt{k}}.$  when $q = 4$. As shown in Figure \ref{fig3a}, Median fails, Stochastic ADMM are close to Ideal SGD and outperforms all the other Byzantine-robust algorithms. When the number of Byzantine worker increases to $q=8$, in Stochastic ADMM, we set  $\lambda = 0.5$, $\beta = 3$, $\alpha_0^k = \frac{1}{60+500\sqrt{k}}$, and $\alpha_i^k = \frac{1}{3+500\sqrt{k}}$. As depicted in Figure \ref{fig3b}, Geometric median and Median fail because the stochastic gradients of regular worker $p$ dominate, such that only one digit can be recognized. Stochastic ADMM and RSA both work well, but Stochastic ADMM converges faster than RSA.
\begin{figure}[ht!]
\centering
\subfigure[$q=4$]{
\label{fig3a}
\includegraphics[width=0.45\textwidth]{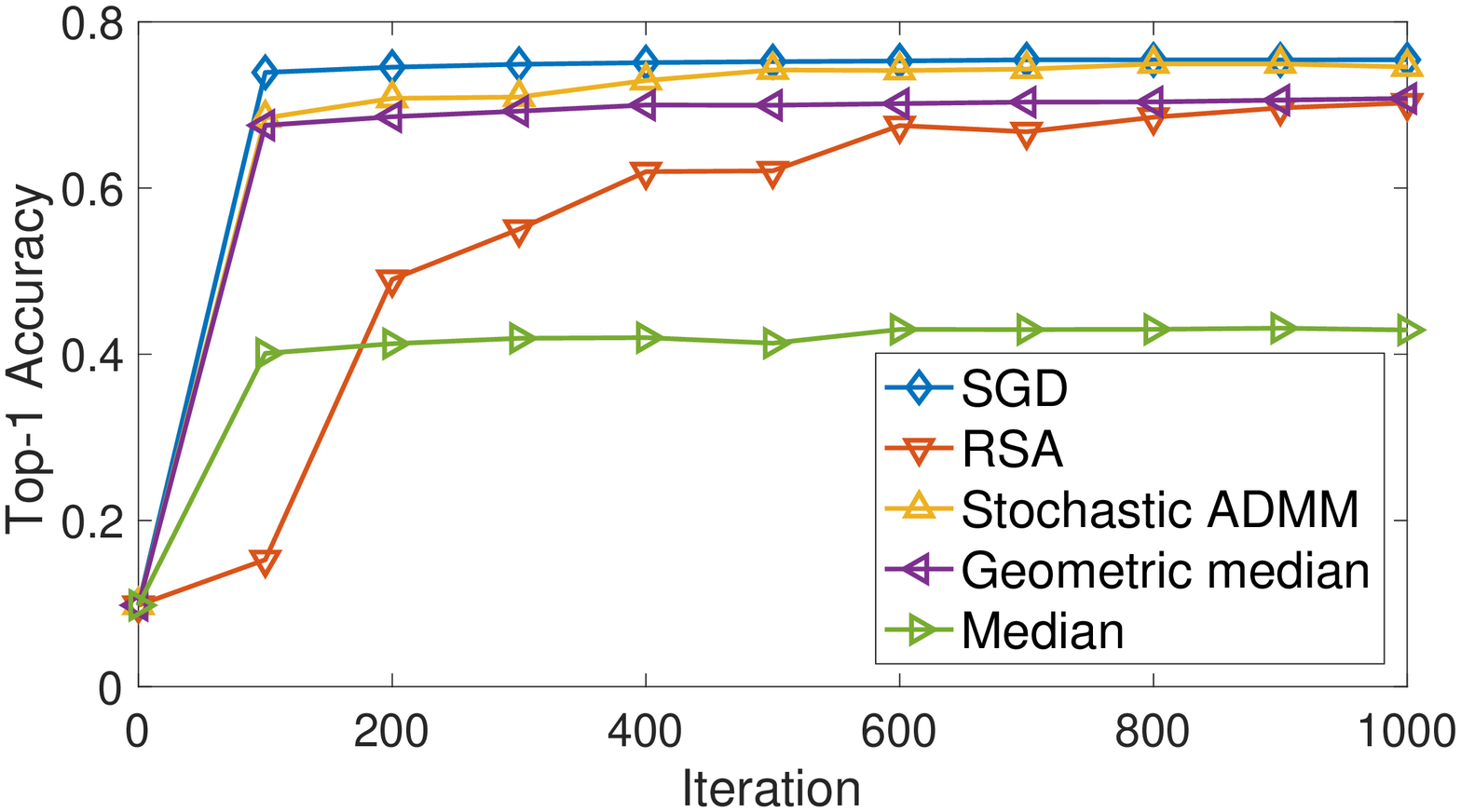}}
\subfigure[$q=8$]{
\label{fig3b}
\includegraphics[width=0.45\textwidth]{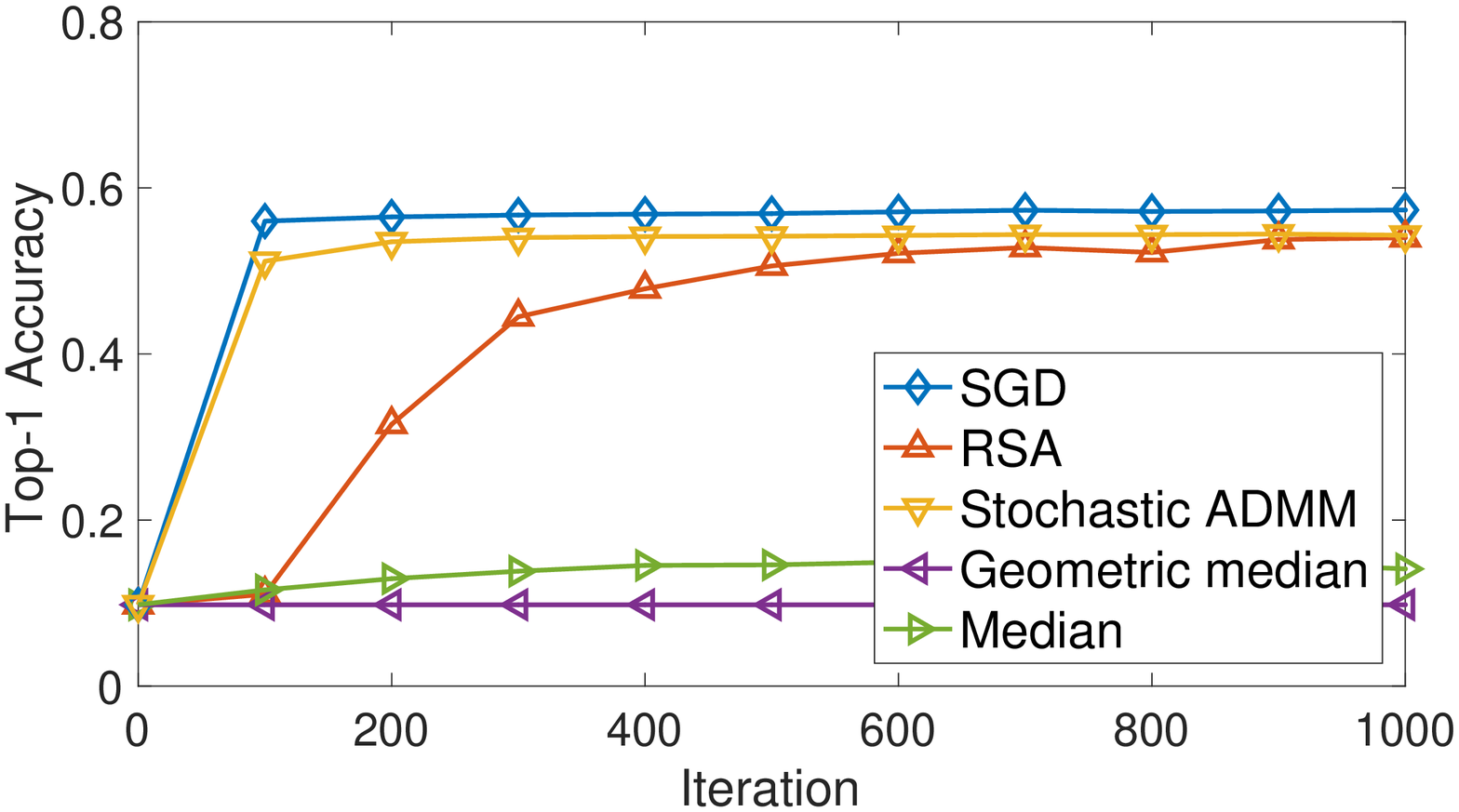}}
\caption{Top-1 accuracy under non-i.i.d. data.}
\label{fig3}
\end{figure}

\section{Conclusions}
\label{sec:6}

We proposed a stochastic ADMM to deal with the distributed learning problem under Byzantine attacks. We considered a TV norm-penalized approximation formulation to handle Byzantine attacks. Theoretically, we proved that the stochastic ADMM converges in expectation to a bounded neighborhood of the optimum at an $O(1/k)$-rate under mild assumptions. 
Numerically, we compared the proposed algorithm with other Byzantine-robust algorithms on two real datasets, showing the competitive performance of the Byzantine-robust stochastic ADMM.

\vspace{1em}

\noindent \textbf{Acknowledgement.} Qing Ling is supported in part by NSF China Grant 61973324, Fundamental Research Funds for the Central Universities, and Guangdong Province Key Laboratory of Computational Science Grant 2020B1212060032. A preliminary version of this paper has appeared in {\it IEEE International Conference on Acoustics, Speech, and Signal Processing}, Barcelona, Spain, May 4--8, 2020.

\clearpage

\appendix
\section{Proof of Proposition \ref{lem:simp}}\label{app:lem1}

The proof of Proposition \ref{lem:simp} relies on the following Lemma.

\begin{lemma}\label{lem:cal}
Let $\{z_1^*,z_2^*\}={\arg\min}_{z_1,z_2\in\mathbb{R}} ~ \lambda|z_1-z_2|+\frac12(z_1-a_1)^2+\frac12(z_2-a_2)^2$, where $\lambda, a_1, a_2 \in \mathbb{R}$. Then \begin{equation}z_1^*+z_2^*=a_1+a_2 \quad \text{and} \quad z_1^*-a_1=\mathrm{proj}_\lambda\left(\frac{a_2-a_1}2\right).\label{lem.in.A}\end{equation}
\end{lemma}
\begin{proof}
Note that $\{z_1^*,z_2^*\}$ together with their difference $\Delta^* = z_1^*-z_2^*$ is also optimal to the bi-level minimization problem
\begin{align}
\min_{\Delta \in \mathbb{R}} & \min_{z_1,z_2\in\mathbb{R}} ~ \lambda|z_1-z_2|+\frac12(z_1-a_1)^2+\frac12(z_2-a_2)^2, \nonumber \\
                                                         & \quad s.t.  \quad z_1-z_2 = \Delta. \nonumber
\end{align}
That is, we first optimize the constrained minimization problem with an artificially imposed constraint $z_1-z_2 = \Delta$ over $\{z_1,z_2\}$, and then optimize over $\Delta$.

For the inner-level constrained minimization problem, from its KKT (Karush-Kuhn-Tucker) conditions we know the minimizer is $\{z_1^*=\frac{a_1+a_2+\Delta}2, z_2^*=\frac{a_1+a_2-\Delta}2\}$, and accordingly, the optimal value is $\lambda |\Delta| + \frac{(\Delta - (a_1-a_2))^2}{4}$. Therefore, for the outer-level unconstrained minimization problem, from its KKT conditions we know the minimizer can be written as $\Delta^* = (a_1-a_2) - \mathrm{proj}_{2\lambda}(a_1-a_2)$. Substituting this result to $\{z_1^*=\frac{a_1+a_2+\Delta^*}2, z_2^*=\frac{a_1+a_2-\Delta}2^*\}$ yields $z_1^*=a_1+\mathrm{proj}_\lambda\left(\frac{a_2-a_1}2\right)$ and $z_2^*=a_2+\mathrm{proj}_\lambda\left(\frac{a_1-a_2}2\right)$. From them we obtain \eqref{lem.in.A} and complete the proof.
\end{proof}

Now we begin to prove Proposition \ref{lem:simp}. Since \eqref{eq8b} is separable with respect to $i$, it is equivalent to
\begin{align}\label{eq9}
&~\left\{z^{k+1}(i,0), z^{k+1}(0,i)\right\}\\
=&~{\arg\min}_{z(i,0), z(0,i)} \lambda\| z(0,i)-z(i,0)\|_1 +\frac{\beta}{2}\big\| z(i,0)-\big(x_0^{k+1} - \frac1\beta\eta^k(i,0) \big) \big\|^2 + \frac{\beta}{2}\big\| z(0,i)- \big(x_i^{k+1}-\frac1\beta\eta^k(0,i)\big)\big\|^2,\notag
\end{align}
According to Lemma \ref{lem:cal}, \eqref{eq9} leads to
\begin{subequations}
\begin{align}
  {z^{k+1}(0,i)+z^{k+1}(i,0)} & = {x_i^{k+1}+x_0^{k+1}}-\frac{\eta^k(0,i)+\eta^k(i,0)}{\beta},  \label{eq10a}\\
  z^{k+1}(i,0)-\big(x_0^{k+1} - \frac1\beta\eta^k(i,0) \big) & = \mathrm{proj}_{\lambda / \beta}\left(\frac{(\eta^k(i,0)-\eta^k(0,i))/\beta + x_i^{k+1}-x_0^{k+1}}2 \right). \label{eq10b}
\end{align}
\end{subequations}

From \eqref{eq8c} and \eqref{eq10a} we have
\begin{equation}\label{equua}
  \eta^{k+1}(0,i)+\eta^{k+1}(i,0) = 0,
\end{equation}
which is regardless of the initialization of $\eta$.
If we further initialize $\eta^0(i,0)=-\eta^0(0,i)$, for simplicity we can define for any $k\ge0$ that $$\eta^k_i:=\eta^k(i,0)=-\eta^k(0,i).$$
With this notation, we rewrite \eqref{eq8c} as
\begin{align}
\eta^{k+1}_i = \eta^{k}_i + \beta(z^{k+1}(i,0)-x_0^{k+1})  = \mathrm{proj}_{\lambda}\left(\eta^k_i+\frac{\beta}{2}(x^{k+1}_i-x^{k+1}_0) \right),\label{eta-upd}
\end{align}
where the last equality is from \eqref{eq10b}.

Next we simplify the updates of $x_i^{k+1}$ and $x_0^{k+1}$ in \eqref{eqxx}. From \eqref{eq8c}, we have
\begin{align}
\beta x_i^k - \beta z^k(0,i) - \eta^k(0,i) &= \eta^{k-1}(0,i) - \eta^k(0,i) - \eta^k(0,i) = 2\eta^k_i-\eta^{k-1}_i,\notag
\end{align}
which simplifies \eqref{eqxx} to
\begin{align}
x^{k+1}_i  &= x_i^k -\alpha_i^k \left( F'(x_i^k,\xi^k_i) + 2\eta^k_i-\eta^{k-1}_i \right), \\
x^{k+1}_0  &= x_0^k -\alpha_0^k \Big( f_0'(x_0^k) - \sum_{i \in \mathcal{R}} (2\eta^k_i-\eta^{k-1}_i ) \Big).
\end{align}
This completes the proof. \qed

\section{Supporting Lemmas}\label{app:opt}
\begin{lemma}[\textbf{Optimality conditions of \eqref{eq5}}]
\label{thm:opt}
The sufficient and necessary optimality conditions of \eqref{eq5} are
\begin{equation}
\label{eqkkt}
\begin{cases}
\mathbb{E}[F'({x_i^*}, \xi_i)]=\eta^*(0,i),\quad
f'_0(x_0^*)=\sum_{i \in \mathcal{R}} \eta^*(i,0),\\
\eta^*(0,i) = - \lambda g_i^*,\quad
\eta^*(i,0) = \lambda g_i^*,\\
z^*(i,0)=x_0^*,\quad
z^*(0,i)=x_i^*,
\end{cases}
\end{equation}
for all $i \in \mathcal{R}$, where $g_i^*\in{sgn}(z^*(0,i)-z^*(i,0))$.
In particular, defining $\eta_i^*:=\eta^*(i,0)$, we have for all $i \in \mathcal{R}$ that
\begin{align}
f'_0(x_0^*) = \sum_{i \in \mathcal{R}} \eta_i^*, \quad
\mathbb{E}[F'({x_i^*}, \xi_i)] = - \eta_i^*, \quad \text{and} \quad \eta_i^*=\lambda g_i^*\in[-\lambda,\lambda]^d. \label{eq:opt}
\end{align}
\end{lemma}

\begin{proof}
The KKT conditions of \eqref{eq5} are given by
\begin{align*}
\begin{cases}
\mathbb{E}[F'({x_i^*}, \xi_i)]-\eta^*(0,i)=0,\quad
f'_0(x_0^*)+\sum_i\big(-\eta^*(i,0)=0,\\
0 \in \lambda \partial_{z^*(0,i)}\|z^*(0,i)-z^*(i,0)\|_1 + \eta^*(0,i),\quad
0 \in \lambda \partial_{z^*(i,0)}\|z^*(0,i)-z^*(i,0)\|_1 + \eta^*(i,0),\\
z^*(i,0)-x_0^*=0,\quad
z^*(0,i)-x_i^*=0,
\end{cases}
\end{align*}
where $\partial_{z^*(0,i)}(\cdot)$ and $\partial_{z^*(i,0)}(\cdot)$ denote the sets of subgradients. Applying the definition of the element-wise sign function $sgn(\cdot)$ yields \eqref{eqkkt}. Further noticing $\eta_i^*:=\eta^*(i,0) = -\eta^*(i,0)$, we obtain \eqref{eq:opt}.
\end{proof}

\begin{lemma}[Theorem 2.1.12, \cite{Nesterov}] \label{lem:nesterov}
With Assumptions \ref{ass1} and \ref{ass2}, for any $x_i$ and $x_0$, it holds
\begin{align}
&\langle F'_i(x_i)-F'_i(x_i^*), x_i-x_i^* \rangle \ge\frac1{\mu_i+L_i}\|F'_i(x_i)-F'_i(x_i^*)\|^2 + \frac{\mu_iL_i}{\mu_i+L_i}\|x_i-x_i^*\|^2,\quad\label{eq:nesi} \\
&\langle f'_0(x_0)-f'_0(x_0^*), x_0-x_0^* \rangle \ge  \frac1{\mu_0+L_0}\|f'_0(x_0)-f'_0(x_0^*)\|^2 + \frac{\mu_0L_0}{\mu_0+L_0}\|x_0-x_0^*\|^2.  \label{eq:nes0}
\end{align}
\end{lemma}

\begin{lemma}\label{lem:eta}
With the simplified updates \eqref{eqxi-update}, \eqref{eqx0-update}, and \eqref{eq11}, it holds for any regular worker $i\in\mathcal{R}$ that
\begin{align}\label{eq:lem:eta}
\langle \eta_i^{k+1}-\eta_i^*, x_0^{k+1}-x_i^{k+1} \rangle\le&~-\frac2{\beta}\langle \eta_i^{k+1}-\eta_i^*,\eta_i^{k+1}-\eta_i^{k} \rangle=\frac1\beta \left( \|\eta_i^{k}-\eta_i^*\|^2 - \|\eta_i^{k+1}-\eta_i^*\|^2 - \|\eta_i^{k+1}-\eta_i^{k}\|^2 \right),\\
\langle \eta_i^{k+1}, x_0^{k+1}-x_i^{k+1} \rangle\le&~-\frac2{\beta}\langle \eta_i^{k+1},\eta_i^{k+1}-\eta_i^{k} \rangle=\frac1\beta \left( \|\eta_i^{k}\|^2 - \|\eta_i^{k+1}\|^2 - \|\eta_i^{k+1}-\eta_i^{k}\|^2 \right).\label{eq:lem:eta2}
\end{align}
\end{lemma}
\begin{proof}
We first show the inequality in \eqref{eq:lem:eta}, and then modify it to prove \eqref{eq:lem:eta2}. The relationships derived during the proof of Lemma \ref{lem:cal} are useful here. Since $\eta^k_i:=\eta^k(i,0)=-\eta^k(0,i)$, the right-hand side of \eqref{eq10b} is exactly $\frac1\beta\eta_i^{k+1}$. Combining this fact with \eqref{eq10a} yields
\begin{equation}
z^{k+1}(i,0)=x_0^{k+1}+\frac1\beta\big( \eta_i^{k+1}-\eta_i^{k} \big),\quad
z^{k+1}(0,i)=x_i^{k+1}-\frac1\beta\big( \eta_i^{k+1}-\eta_i^{k} \big).\label{eq:z^k+1}
\end{equation}
Recall that in \eqref{eq9}, we minimize the function
$$\lambda\| z(0,i)-z(i,0)\|_1 +\frac{\beta}{2}\big\| z(i,0)-\big(x_0^{k+1} - \frac1\beta\eta^k_i \big) \big\|^2 + \frac{\beta}{2}\big\| z(0,i)- \big(x_i^{k+1}+\frac1\beta\eta^k_i\big)\big\|^2,$$
with respect to $\big[z(i,0); z(0,i)\big]\in\mathbb{R}^{2d}$. From the first order optimality condition, there exists a subgradient of $\lambda\|z(0,i)-z(i,0)\|_1$ at $\big[z^{k+1}(0,i);z^{k+1}(i,0)\big]$, denoted as $[\lambda g_i^{k+1};-\lambda g_i^{k+1}]\in[-\lambda,\lambda]^{2d}$, such that
\begin{equation}\label{eq:g^k+1}
0=\lambda g_i^{k+1} + \beta\big( z^{k+1}(0,i)- (x_i^{k+1} + \frac1\beta\eta^k_i )\big)=\lambda g_i^{k+1} - \eta^{k+1}.
\end{equation}
Applying the definition of subgradient of $\lambda\|z(0,i)-z(i,0)\|_1$ at points $\big[z^{k+1}(0,i);z^{k+1}(i,0)\big]$ and $\big[z^*(0,i);z^*(i,0)\big]$ gives
\begin{align}
\lambda\|z^*(0,i)-z^*(i,0)\|_1\ge &~\lambda\|z^{k+1}(0,i)-z^{k+1}(i,0)\|_1 \notag\\
&~+ \big\langle [(\lambda g_i^{k+1};-\lambda g_i^{k+1}]), \big[z^*(0,i)-z^{k+1}(0,i); z^*(i,0) - z^{k+1}(i,0)\big] \big\rangle,\label{eq:subgrad1}\\
\lambda\|z^{k+1}(0,i)-z^{k+1}(i,0)\|_1\ge&~ \lambda\|z^*(0,i)-z^*(i,0)\|_1 \notag\\
&~+ \big\langle [\lambda g_i^*;-\lambda g_i^*], \big[z^{k+1}(0,i)-z^*(0,i); z^{k+1}(i,0) - z^*(i,0)\big] \big\rangle,\label{eq:subgrad2}
\end{align}
where $g_i^*=\frac{\eta_i^*}\lambda $ is defined in Lemma \ref{thm:opt}. Summing up \eqref{eq:subgrad1} and \eqref{eq:subgrad2}, we have
\begin{equation*}
0\ge \big\langle \lambda g_i^{k+1}-\lambda g_i^*, z^{k+1}(i,0)-z^{k+1}(0,i) \big\rangle
=\big\langle \eta^{k+1}-\eta^*,x_0^{k+1}-x_i^{k+1}+\frac2\beta\big( \eta_i^{k+1}-\eta_i^{k} \big) \big\rangle,
\end{equation*}
where the last equality comes from \eqref{eq:z^k+1} and \eqref{eq:g^k+1}. Rearranging the terms gives \eqref{eq:lem:eta}.

Further note that $[\lambda g_i^*;-\lambda g_i^*]$ in \eqref{eq:subgrad2} can be replaced by any subgradient of $\lambda\|z(0,i)-z(i,0)\|_1$ at $\big[z^*(0,i);z^*(i,0)\big]$. We hence replace it by $[0;0]$, and then obtain \eqref{eq:lem:eta2}.
\end{proof}

\section{Proof of Theorem \ref{thm:conv}}\label{app:conv}

\noindent{\textbf{Restatement of Theorem \ref{thm:conv}}.}
Suppose Assumptions \ref{ass1}, \ref{ass2}, and \ref{ass:var} hold. let $\lambda\ge\lambda_0$ and the stepsizes be
$$\alpha_0^k=\min\left\{\frac1{ck+m\beta},A\right\}, \quad \alpha_i^k=\min\left\{\frac1{ck+\beta}, A\right\},\ \forall i\in\mathcal{R},$$ for some positive constants $c<\min\left\{\frac{\mu_0L_0}{\mu_0+L_0}, \frac{\mu_iL_i}{\mu_i+L_i}: i\in\mathcal{R} \right\}$, $\beta$, and ${A}\le\min\left\{\frac1{\mu_0+L_0}, \frac1{\mu_i+L_i}: i\in\mathcal{R} \right\}$. Define constants
\begin{align*}
c_0&=\min\bigg\{c + (m-1)\beta, \frac{\mu_0L_0}{\mu_0+L_0}, \frac{\mu_iL_i}{\mu_i+L_i}: i \in \mathcal{R} \bigg\},\\
c_1&=\frac{9d(\mu_0+L_0)}{\mu_0L_0}\lambda^2q^2,\\
c_2&= \left[2(4r+3q)^2+64r+16(m-1)\beta\left(\frac{r}{2\beta}+\sum_i\bigg(\frac1{\mu_i}+\frac1{L_i}\bigg)
\right)\right]\lambda^2d+4\sum_i\delta_i^2.
\end{align*}
Also denote $V^k=\mathbb{E} \|x_0^k-x_0^*\|^2+\sum_{i \in \mathcal{R}} \big(\mathbb{E} \|x_i^k-x_i^*\|^2+\frac{2\alpha_i^{k-1}}\beta\|\eta_i^{k-1}-\eta_i^*\|^2\big)$. Then we have
\begin{align}
V^{k+1} \le \begin{cases} V^k(1-c_0\alpha^k_0) + c_1\alpha^k_0 + c_2(\alpha^k_i)^2, \quad & k>k_0,\\
 V^k(1-c_0A) + {A} c_1 + {A}^2 c_2, \quad & k\le k_0,\end{cases}\label{eq:iter}
\end{align}
where $k_0=\min\{k:\frac1{ck+\beta}<A\}$.
Consequently, it holds
\begin{equation}
V^k\le\begin{cases}
V^0 (1-cA)^{k+1}+ \frac{c_1+{A}c_2}{c_0}, \quad &k \le k_0,\\
\frac{C}{c(k-1)+m\beta}+\frac{c_1}{c_0}, \quad &k >k_0,
\end{cases}
\label{eq:Vk-bound-app}
\end{equation}
with
\begin{equation*}
C=\max\left\{\left(\frac{ck_0+m\beta}{ck_0+\beta}\right)^2\frac{c_2}{c_0-c}, (ck_0+m\beta)\left(V^0(1-c_0A)^{k_0+1}+\frac{Ac_2}{c_0}\right)\right\}.
\end{equation*}
\begin{proof}
Recall that the updates satisfy
\begin{align}
x_0^{k+1}= &~x_0^k - \alpha_0^k\bigg( f'_0(x_0^k) - \sum_{i \in \mathcal{R}} \big(2\eta^k_i-\eta^{k-1}_i\big) - \sum_{j \in \mathcal{B}} \big(2\eta^k_j-\eta^{k-1}_j \big)\bigg),\label{x0-adm-sim}\\
x_i^{k+1}=&~x_i^k - \alpha_i^k\bigg( F'(x_i^k,\xi_i^k) + \big(2\eta^k_i-\eta^{k-1}_i\big) \bigg), \forall i \in \mathcal{R}, \label{xi-adm-sim}\\
\eta_i^{k+1}=&~\mathrm{proj}_{\lambda}\left(\eta^k_i+\frac{\beta}{2}(x^{k+1}_i-x^{k+1}_0) \right)\in[-\lambda,\lambda]^d, \forall i \in \mathcal{R} \quad \text{and} \quad \eta_j^{k+1}\in[-\lambda,\lambda]^d, \forall j \in \mathcal{B}.\label{eta-range}
\end{align}

{\bf Step 1.} At the master side, we have
\begin{align}
\mathbb{E}  \|x_0^{k+1}-x_0^*\|^2\overset{\eqref{x0-adm-sim}}=&~ \mathbb{E} \|x_0^k-x_0^*\|^2\notag
 + (\alpha_0^k)^2 \mathbb{E}\bigg\| f'_0(x_0^k) - \sum_{i \in \mathcal{R}} \big(2\eta^k_i-\eta^{k-1}_i\big) - \sum_{j \in \mathcal{B}} \big(2\eta^k_j-\eta^{k-1}_j \big) \bigg\|^2\\
&~-2 \alpha_0^k \mathbb{E}\bigg\langle f'_0(x_0^k) - \sum_{i \in \mathcal{R}} \big(2\eta^k_i-\eta^{k-1}_i\big) - \sum_{j \in \mathcal{B}} \big(2\eta^k_j-\eta^{k-1}_j\big), x_0^k-x_0^* \bigg\rangle. \label{eq:expand0}
\end{align}
For the second term in \eqref{eq:expand0}, the inequality $\|a+b\|^2\le 2\|a\|^2+2\|b\|^2$ gives
\begin{align}
&~\mathbb{E}\bigg\| f'_0(x_0^k) - \sum_{i \in \mathcal{R}} \big(2\eta^k_i-\eta^{k-1}_i\big) - \sum_{j \in \mathcal{B}} \big(2\eta^k_j-\eta^{k-1}_j \big) \bigg\|^2\notag\\
\overset{\eqref{eq:opt}}=&~\mathbb{E}\bigg\| f'_0(x_0^k) - f'_0(x_0^*) - \sum_{i \in \mathcal{R}} \big(2\eta^k_i-\eta^{k-1}_i-\eta^*_i\big) - \sum_{j \in \mathcal{B}} \big(2\eta^k_j-\eta^{k-1}_j \big) \bigg\|^2\notag\\
\le&~2\mathbb{E}\big\| f'_0(x_0^k) - f'_0(x_0^*)\big\|^2 + 2\mathbb{E}\bigg\| \sum_{i \in \mathcal{R}} \big(2\eta^k_i-\eta^{k-1}_i-\eta^*_i\big) + \sum_{j \in \mathcal{B}} \big(2\eta^k_j-\eta^{k-1}_j \big) \bigg\|^2\notag\\
\overset{\eqref{eta-range}}\le&~2\mathbb{E}\big\| f'_0(x_0^k) - f'_0(x_0^*)\big\|^2 + 2(4r+3q)^2\lambda^2d.\label{eq:x0-2}
\end{align}
Applying the inequality $2\langle a,b \rangle\le \epsilon \|a\|^2 + \frac1\epsilon \|b\|^2$ to the third term in \eqref{eq:expand0} with $\epsilon=\frac{\mu_0L_0}{\mu_0+L_0}$ yields \begin{align}
&-2\mathbb{E}\bigg\langle f'_0(x_0^k) - \sum_{i \in \mathcal{R}} \big(2\eta^k_i-\eta^{k-1}_i\big) - \sum_{j \in \mathcal{B}} \big(2\eta^k_j-\eta^{k-1}_j\big), x_0^k-x_0^* \bigg\rangle\notag\\
\overset{\eqref{eq:opt}}=&-2\mathbb{E}\big\langle f'_0(x_0^k) - f'_0(x_0^*), x_0^k-x_0^* \big\rangle + 2 \sum_{i \in \mathcal{R}}\mathbb{E}\big\langle 2\eta^k_i-\eta^{k-1}_i-\eta_i^*, x_0^k-x_0^* \big\rangle +
2\mathbb{E}\bigg\langle \sum_{j \in \mathcal{B}} 2\eta^k_j-\eta^{k-1}_j, x_0^k-x_0^* \bigg\rangle
\notag\\
\overset{\eqref{eq:nes0}}\le&-\frac2{\mu_0+L_0}\mathbb{E}\|f'_0(x_0^k)-f'_0(x_0^*)\|^2 - \frac{2\mu_0L_0}{\mu_0+L_0}\mathbb{E}\|x_0^k-x_0^*\|^2 +2 \sum_{i \in \mathcal{R}}\mathbb{E}\big\langle 2\eta^k_i-\eta^{k-1}_i-\eta_i^*, x_0^k-x_0^* \big\rangle \notag\\
&+ \frac{\mu_0L_0}{\mu_0+L_0} \mathbb{E}\|x_0^k-x_0^*\|^2 + \frac{\mu_0+L_0}{\mu_0L_0} \mathbb{E}\bigg\|\sum_{j \in \mathcal{B}} \big(2\eta^k_j-\eta^{k-1}_j\big)\bigg\|^2\notag\\
\le&-\frac2{\mu_0+L_0}\mathbb{E}\|f'_0(x_0^k)-f'_0(x_0^*)\|^2 - \frac{\mu_0L_0}{\mu_0+L_0} \mathbb{E}\|x_0^k-x_0^*\|^2 + \frac{\mu_0+L_0}{\mu_0L_0}(3\lambda q)^2d \notag\\
&+ 2 \sum_{i \in \mathcal{R}}\mathbb{E}\big\langle 2\eta^k_i-\eta^{k-1}_i-\eta_i^*, x_0^k-x_0^* \big\rangle.\label{eq:x0-3}
\end{align}
Substituting \eqref{eq:x0-2} and \eqref{eq:x0-3} into \eqref{eq:expand0} gives
\begin{align}
\mathbb{E} \|x_0^{k+1}-x_0^*\|^2 \le&~ \mathbb{E} \|x_0^k-x_0^*\|^2\left( 1-\alpha_0^k\frac{\mu_0L_0}{\mu_0+L_0} \right) - \mathbb{E}\big\| f'_0(x_0^k) - f'_0(x_0^*)\big\|^2 2\alpha_0^k \left( \frac1{\mu_0+L_0}-\alpha_0^k \right)\notag \\
&~+ 2\lambda^2d(4r+3q)^2(\alpha_0^k)^2 + \frac{9\lambda^2d(\mu_0+L_0)}{\mu_0L_0}q^2\alpha_0^k  + 2\alpha_0^k \sum_{i \in \mathcal{R}}\mathbb{E}\big\langle 2\eta^k_i-\eta^{k-1}_i-\eta_i^*, x_0^k-x_0^* \big\rangle\notag\\
\le&~ \mathbb{E} \|x_0^k-x_0^*\|^2\left( 1-\alpha_0^k \frac{\mu_0L_0}{\mu_0+L_0} \right) + 2\lambda^2d(4r+3q)^2(\alpha^k_0)^2 + \frac{9\lambda^2d(\mu_0+L_0)}{\mu_0L_0}q^2\alpha^k_0 \notag\\
&+ 2\alpha_0^k \sum_{i \in \mathcal{R}} \mathbb{E}\big\langle 2\eta^k_i-\eta^{k-1}_i-\eta_i^*, x_0^k-x_0^* \big\rangle, \label{eq:step1}
\end{align}
where the last inequality comes from that $\alpha_0^k\le{A}\le\frac1{\mu_0+L_0}$.

{\bf Step 2.} Accordingly, at the regular worker side, we have for any $i \in \mathcal{R}$ that
\begin{align}
\mathbb{E} \|x_i^{k+1}-x_i^*\|^2 \overset{\eqref{xi-adm-sim}}=&~ \mathbb{E} \|x_i^k-x_i^*\|^2 + (\alpha_i^k)^2\mathbb{E}\bigg\| F'(x_i^k,\xi_i^k) + \big(2\eta^k_i-\eta^{k-1}_i\big) \bigg\|^2\label{eq:expandi}\\
&~-2\alpha_i^k \mathbb{E}\bigg\langle F'(x_i^k,\xi_i^k) + \big(2\eta^k_i-\eta^{k-1}_i\big), x_i^k-x_i^* \bigg\rangle. \notag
\end{align}
For the second term in \eqref{eq:expandi}, the inequality $\|a+b\|^2\le 2\|a\|^2+2\|b\|^2$ gives that
\begin{align}
 \mathbb{E}\bigg\| F'(x_i^k,\xi_i^k) + \big(2\eta^k_i-\eta^{k-1}_i\big) \bigg\|^2 \overset{\eqref{eq:opt}}= &~ \mathbb{E}\bigg\|F'(x_i^k,\xi_i^k) - F'_i(x_i^k) + F'_i(x_i^k) -F'_i(x_i^*) + \big(2\eta^k_i-\eta^{k-1}_i-\eta_i^*\big) \bigg\|^2\notag\\
\le&~2\mathbb{E}\big\| F'_i(x_i^k) -F'_i(x_i^*)\big\|^2 + 4\mathbb{E}\|F'(x_i^k,\xi_i^k) - F'_i(x_i^k)\|^2 + 4\mathbb{E}\|2\eta^k_i-\eta^{k-1}_i-\eta^*_i\|^2\notag\\
\le&~2\mathbb{E}\big\| F'_i(x_i^k) -F'_i(x_i^*)\big\|^2 + 4\delta_i^2 + 64\lambda^2d,\label{eq:xi-2}
\end{align}
where the last inequality comes from \eqref{eq:ass-var} and \eqref{eta-range}. Then the third term in \eqref{eq:expandi} can be upper-bounded as
\begin{align}
&-2\mathbb{E}\bigg\langle F'(x_i^k,\xi_i^k) + \big(2\eta^k_i-\eta^{k-1}_i\big), x_i^k-x_i^* \bigg\rangle\notag\\
=&-2\mathbb{E}\bigg\langle F'_i(x_i^k) + \big(2\eta^k_i-\eta^{k-1}_i\big), x_i^k-x_i^* \bigg\rangle\notag\\
\overset{\eqref{eq:opt}}=&-2\mathbb{E}\big\langle F'_i(x_i^k) - F'_i(x_i^*), x_i^k-x_i^* \big\rangle - 2\mathbb{E}\langle 2\eta^k_i-\eta^{k-1}_i-\eta_i^*, x_i^k-x_i^* \rangle\notag\\
\overset{\eqref{eq:nesi}}\le&-\frac2{\mu_i+L_i}\mathbb{E}\|F'_i(x_i^k) - F'_i(x_i^*)\|^2 - \frac{2\mu_iL_i}{\mu_i+L_i}\mathbb{E}\|x_i^k-x_i^*\|^2 - 2\mathbb{E}\langle 2\eta^k_i-\eta^{k-1}_i-\eta_i^*, x_i^k-x_i^* \rangle,\label{eq:xi-3}
\end{align}
where the first equality comes from taking expectation of the conditional expectation; that is, $\mathbb{E} x = \mathbb{E}\big[\mathbb{E}[x|\mathcal{F}_{k-1}]\big]$ with $\mathcal{F}_{k-1}$ denoting the sigma-field generated by $\{\xi_i^{l-1},\eta_j^l: l\le k, i\in\mathcal{R}, j\in\mathcal{B}\}$.

Substituting \eqref{eq:xi-2} and \eqref{eq:xi-3} into \eqref{eq:expandi} gives
\begin{align}
\mathbb{E} \|x_i^{k+1}-x_i^*\|^2
\le&~ \mathbb{E} \|x_i^k-x_i^*\|^2\left( 1 - \alpha_i^k\frac{\mu_iL_i}{\mu_i+L_i} \right) - \mathbb{E}\big\| F'_i(x_i^k) - F'_i(x_i^*)\big\|^2 2\alpha_i^k\left( \frac1{\mu_i+L_i}-\alpha_i^k \right) \notag\\
&~+(4\delta_i^2 + 64\lambda^2d)(\alpha_i^k)^2 - 2\alpha_i^k\mathbb{E}\langle 2\eta^k_i-\eta^{k-1}_i-\eta_i^*, x_i^k-x_i^* \rangle\notag\\
\le&~ \mathbb{E} \|x_i^k-x_i^*\|^2\left( 1-\alpha_i^k \frac{\mu_iL_i}{\mu_i+L_i} \right)+(4\delta_i^2 + 64\lambda^2d)(\alpha_i^k)^2 -2\alpha_i^k\mathbb{E}\langle 2\eta^k_i-\eta^{k-1}_i-\eta_i^*, x_i^k-x_i^* \rangle, \label{eq:step2}
\end{align}
where the last inequality comes from that $\alpha_i^k\le{A}\le \frac1{\mu_i+L_i}$.

{\bf Step 3.} Now combine \eqref{eq:step1} with \eqref{eq:step2}. Using the notation $V^k=\mathbb{E} \|x_0^k-x_0^*\|^2+\sum_i \mathbb{E} \big(\|x_i^k-x_i^*\|^2+\frac{2\alpha_i^{k-1}}\beta\|\eta_i^{k-1}-\eta_i^*\|^2\big)$, we have
\begin{align}
V^{k+1}\le &~\mathbb{E} \|x_0^k-x_0^*\|^2\left( 1-\alpha_0^k \frac{\mu_0L_0}{\mu_0+L_0} \right) +\sum_{i\in\mathcal{R}} \mathbb{E} \|x_i^k-x_i^*\|^2\left( 1-\alpha_i^k \frac{\mu_iL_i}{\mu_i+L_i} \right)\label{eq:iter-general}\\
&~+ \lambda^2d\bigg[2(4r+3q)^2(\alpha_0^k)^2 + 64\sum_{i\in\mathcal{R}}(\alpha_i^k)^2\bigg] + 4\sum_{i\in\mathcal{R}}\delta_i^2(\alpha_i^k)^2 + \frac{9\lambda^2d(\mu_0+L_0)}{\mu_0L_0}q^2\alpha^k_0 \notag\\
&~+\sum_{i\in\mathcal{R}} \frac{2\alpha_i^k}\beta\mathbb{E}\|\eta_i^k-\eta_i^*\|^2 - 2\sum_{i\in\mathcal{R}}\mathbb{E}\big\langle 2\eta^k_i-\eta^{k-1}_i-\eta_i^*, \alpha_i^k(x_i^k-x_i^*)-\alpha_0^k(x_0^k-x_0^*) \big\rangle \notag.
\end{align}
For the last term in \eqref{eq:iter-general}, notice that
\begin{align}
&~-\mathbb{E}\big\langle 2\eta^k_i-\eta^{k-1}_i-\eta_i^*, \alpha_i^k(x_i^k-x_i^*)-\alpha_0^k(x_0^k-x_0^*) \big\rangle\notag\\
=&~-\alpha_i^k\mathbb{E}\big\langle 2\eta^k_i-\eta^{k-1}_i-\eta_i^*, x_i^k-x_0^k \big\rangle - (\alpha_i^k-\alpha_0^k)\mathbb{E}\big\langle 2\eta^k_i-\eta^{k-1}_i-\eta_i^*, x_0^k-x_0^* \big\rangle\notag\\
=&~-\alpha_0^k\mathbb{E}\big\langle \eta^k_i-\eta_i^*, x_i^k-x_0^k \big\rangle - \alpha_0^k\mathbb{E}\big\langle \eta^k_i-\eta_i^{k-1}, x_i^k-x_0^k \big\rangle - (\alpha_i^k-\alpha_0^k)\mathbb{E}\big\langle 2\eta^k_i-\eta^{k-1}_i-\eta_i^*, x_i^k-x_i^* \big\rangle\label{eq:step3-2},
\end{align}
where the first equality comes from Corollary \ref{Coro:1} that $x_i^*=x_0^*$.
For the first term in \eqref{eq:step3-2}, Lemma \ref{lem:eta} suggests that
\begin{equation}\label{eq:step3-p1}
-\alpha_0^k\mathbb{E}\big\langle \eta^k_i-\eta_i^*, x_i^k-x_0^k \big\rangle\le \frac{\alpha_0^k}\beta \left( \|\eta_i^{k-1}-\eta_i^*\|^2 - \|\eta_i^{k}-\eta_i^*\|^2 - \|\eta_i^{k}-\eta_i^{k-1}\|^2 \right) \le \frac{\alpha_0^k}\beta \left( \|\eta_i^{k-1}-\eta_i^*\|^2 - \|\eta_i^{k}-\eta_i^*\|^2\right).
\end{equation}
For the second therm in \eqref{eq:step3-2}, the projection operator in the $\eta_i$-update gives that
\begin{equation}\label{eq:step3-p2}
\big\langle \eta^k_i-\eta_i^{k-1}, x_i^k-x_0^k \big\rangle=\big\langle \mathrm{proj}_{\lambda}\big(\eta^{k-1}_i+\frac{\beta}{2}(x^{k}_i-x^{k}_0) \big)-\eta_i^{k-1}, x_i^k-x_0^k \big\rangle\ge0,
\end{equation}
provided $\eta_i^{k-1}\in[-\lambda,\lambda]^d$. For the third term in \eqref{eq:step3-2}, we apply the equality $2\langle a,b \rangle\le \frac1{\epsilon_i} \|a\|^2 + {\epsilon_i} \|b\|^2$ with ${\epsilon_i}= \frac{\mu_iL_i}{\mu_i+L_i}$ to obtain
\begin{align}
- (\alpha_i^k-\alpha_0^k)\mathbb{E}\big\langle 2\eta^k_i-\eta^{k-1}_i-\eta_i^*, x_i^k-x_i^* \big\rangle
\le&~ \frac{\alpha_i^k-\alpha_0^k}2\left(\frac{\mu_i+L_i}{\mu_iL_i}(4\lambda)^2d + \frac{\mu_iL_i}{\mu_i+L_i}\mathbb{E}\|x_i^k-x_i^*\|^2\right).\label{eq:step3-p3}
\end{align}

Therefore, applying the bounds in \eqref{eq:step3-p1}, \eqref{eq:step3-p2}, and \eqref{eq:step3-p3} to \eqref{eq:step3-2}, we get
\begin{align}
&~-2\mathbb{E}\big\langle 2\eta^k_i-\eta^{k-1}_i-\eta_i^*, \alpha_i^k(x_i^k-x_i^*)-\alpha_0^k(x_0^k-x_0^*) \big\rangle\notag\\
\le&~\frac{2\alpha_0^k}\beta \left( \mathbb{E}\|\eta_i^{k-1}-\eta_i^*\|^2 - \mathbb{E}\|\eta_i^{k}-\eta_i^*\|^2\right) + (\alpha_i^k-\alpha_0^k)\left(\frac{\mu_i+L_i}{\mu_iL_i}16\lambda^2d + \frac{\mu_iL_i}{\mu_i+L_i}\mathbb{E}\|x_i^k-x_i^*\|^2\right)\notag\\
\le&~\frac{2\alpha_0^k}\beta\mathbb{E} \|\eta_i^{k-1}-\eta_i^*\|^2 - \frac{2\alpha_i^k}\beta\mathbb{E}\|\eta_i^{k}-\eta_i^*\|^2 + (\alpha_i^k-\alpha_0^k)\left( (\frac1{2\beta} + \frac{\mu_i+L_i}{\mu_iL_i})16\lambda^2d + \frac{\mu_iL_i}{\mu_i+L_i}\mathbb{E}\|x_i^k-x_i^*\|^2\right).\label{eq:step3-3}
\end{align}
Consequently \eqref{eq:iter-general} becomes
\begin{align}
V^{k+1}\le &~\mathbb{E} \|x_0^k-x_0^*\|^2\left( 1-\alpha^k_0\frac{\mu_0L_0}{\mu_0+L_0} \right) +\sum_{i\in\mathcal{R}} \mathbb{E} \|x_i^k-x_i^*\|^2\left( 1-\frac{\alpha_i^k+\alpha_0^k}2\frac{\mu_iL_i}{\mu_i+L_i}\right) + \sum_{i\in\mathcal{R}} \frac{\alpha_0^k}{\alpha_i^{k-1}}\frac{2\alpha_i^{k-1}}\beta\mathbb{E}\|\eta_i^{k-1}-\eta_i^*\|^2\notag\\
&~+ \lambda^2d\bigg[2(4r+3q)^2(\alpha_0^k)^2 + 64 \sum_{i\in\mathcal{R}} (\alpha_i^k)^2 + 16(\sum_{i\in\mathcal{R}}\frac{\mu_i+L_i}{\mu_i L_i} + \frac{r}{2\beta})(\alpha_i^k-\alpha_0^k)\bigg] + 4 \sum_{i\in\mathcal{R}} \delta_i^2(\alpha_i^k)^2 + \frac{9\lambda^2d(\mu_0+L_0)}{\mu_0L_0}q^2\alpha^k_0  \notag\\
\le&~ V^k(1-c_0\alpha^k_0) + c_1\alpha^k_0 + c_2(\alpha^k_i)^2,
\end{align}
where the last inequality comes from the upper-bound of $c_0$. Plugging in the choices of stepsizes, we obtain \eqref{eq:iter}.

{\bf Step 4.} Now we iteratively uses \eqref{eq:iter} to derive the $O(1/k)$-convergence of $V^k$. First, when $k\le k_0$, it holds
\begin{align*}
V^{k+1}\le& V^0 (1-c_0A)^{k+1} + ({A} c_1+{A}^2c_2) \big(1+(1-c_0A)+\ldots+(1-c_0A)^k\big)\\
\le &V^0 (1-c_0A)^{k+1}+ \frac{c_1+{A}c_2}{c_0}.
\end{align*}
For any $k\ge k_0 +1$, initially it holds $V^{k_0+1}\le\frac{C}{ck_0+m\beta}+\frac{c_1}{c_0}$ from the definition of $C$.
By deduction, if \eqref{eq:Vk-bound-app} holds for $k$, then
\begin{align*}
V^{k+1}\overset{\eqref{eq:iter}}\le&~\bigg(\frac{C}{c(k-1)+m\beta}+\frac{c_1}{c_0}\bigg)\bigg(1-\frac{c_0}{ck+m\beta}\bigg) + \frac{c_1}{ck+m\beta} + \frac{c_2}{(ck+\beta)^2}\\
=&~\frac{C}{ck+m\beta} - \frac{c_0-c}{c(k-1)+m\beta}\frac{C}{ck+m\beta}+ \frac{c_2}{(ck+\beta)^2} +\frac{c_1}{c_0}\\
\le&~\frac{C}{ck+m\beta} +\frac{c_1}{c_0} + \frac1{(ck+\beta)^2} \bigg({c_2} - \frac{(ck+\beta)^2(c_0-c)}{(ck+m\beta)^2}C\bigg)\\
\le&~\frac{C}{ck+m\beta} +\frac{c_1}{c_0},
\end{align*}
where the last inequality is from $C\ge(\frac{ck_0+m\beta}{ck_0+\beta})^2\frac{c_2}{c_0-c}$. This completes the proof.
\end{proof}

\section{$O(1/\sqrt{k})$-ergodic convergence}\label{app:sqk}
\begin{theorem}
Suppose Assumptions \ref{ass1}, \ref{ass2}, and \ref{ass:var} hold. Let $\lambda\ge\lambda_0$ and the stepsizes be
$$\alpha_0^k=\min\left\{\frac1{\bar{c}\sqrt{k}+m\bar\beta},\bar{A}\right\}, \quad \alpha_i^k=\min\left\{\frac1{\bar{c}\sqrt{k}+\bar\beta}, \bar{A}\right\},\ \forall i\in\mathcal{R}, $$ for some positive constants $\bar{c}$, $\bar\beta$, and $\bar{A}\le\min\left\{\frac{\mu_0}{4L_0^2}, \frac{\mu_i}{2L_i^2+ (m-1)\beta c}: i\in\mathcal{R} \right\}$.
Then the proposed algorithm converges in the ergodic sense that
\begin{equation}
\sum_{i\in\mathcal{R}} \mathbb{E}[F(\bar{x}_i^k,\xi_i)]+f_0(\bar{x}_0^k) - \underbrace{\min_{\tilde{x}}\left( \sum_{i\in\mathcal{R}} \mathbb{E}[F({\tilde{x}},\xi_i)]+f_0({\tilde{x}})\right)}_{\text{Our goal in } \eqref{eq2}} \le m\tilde{c}_1 + {O}\left(\frac{\log k}{\sqrt{k}}\right),
\end{equation}
where $\bar{x}_i^k=\sum_{l=1}^k\frac{\alpha^lx_i^l}{\sum_{l'=1}^k\alpha^{l'}}, \bar{x}_0^k=\sum_{l=1}^k\frac{\alpha^lx_0^l}{\sum_{l'=1}^k\alpha^{l'}}$ are the weighted average variables, and the constant $\bar{c}_1=\frac{18d}{\mu_0}\lambda^2q^2$.
\end{theorem}
\begin{proof}
{\bf Step 1.} At the master side, we still have \eqref{eq:expand0}, in the form of
\begin{align}
\mathbb{E}  \|x_0^{k+1}-x_0^*\|^2=&~ \mathbb{E} \|x_0^k-x_0^*\|^2
 + (\alpha_0^k)^2 \mathbb{E}\bigg\| f'_0(x_0^k) - \sum_{i\in\mathcal{R}} \big(2\eta^k_i-\eta^{k-1}_i\big) - \sum_{j\in\mathcal{B}} \big(2\eta^k_j-\eta^{k-1}_j \big) \bigg\|^2 \notag\\
&~-2 \alpha_0^k \mathbb{E}\bigg\langle f'_0(x_0^k) - \sum_{i\in\mathcal{R}} \big(2\eta^k_i-\eta^{k-1}_i\big) - \sum_{j\in\mathcal{B}} \big(2\eta^k_j-\eta^{k-1}_j\big), x_0^k-x_0^* \bigg\rangle. \label{sqk-step1-0}
\end{align}
For the second term in \eqref{sqk-step1-0}, we also have \eqref{eq:x0-2}, which can be further bounded by
\begin{align}
&~\mathbb{E}\bigg\| f'_0(x_0^k) - \sum_{i\in\mathcal{R}} \big(2\eta^k_i-\eta^{k-1}_i\big) - \sum_{j\in\mathcal{B}} \big(2\eta^k_j-\eta^{k-1}_j \big) \bigg\|^2\le2L_0^2\mathbb{E}\| x_0^k  - x_0^*\|^2 + 2(4r+3q)^2\lambda^2d.\label{sqk-step1-1}
\end{align}
Here we use the fact that $f_0$ has Lipschitz continuous gradients in Assumption \ref{ass2}. In addition, using the fact that $f_0$ is strongly convex in Assumption \ref{ass1} leads to
\begin{equation}
\langle f_0'(x_0^k), x_0^k-x_0^*\rangle \ge f_0(x_0^k)-f_0(x_0^*) + \frac{\mu_0}2\|x_0^k-x_0^*\|^2.\label{sc}
\end{equation}
Applying the inequality $2\langle a,b \rangle\le \epsilon \|a\|^2 + \frac1\epsilon \|b\|^2$ to the third term in \eqref{sqk-step1-0} with $\epsilon=\frac{\mu_0}2$ yields
\begin{align}
&-2\mathbb{E}\bigg\langle f'_0(x_0^k) - \sum_{i\in\mathcal{R}} \big(2\eta^k_i-\eta^{k-1}_i\big) - \sum_{j\in\mathcal{B}} \big(2\eta^k_j-\eta^{k-1}_j\big), x_0^k-x_0^* \bigg\rangle\notag\\
=&-2\mathbb{E}\big\langle f'_0(x_0^k) , x_0^k-x_0^* \big\rangle + 2\mathbb{E}\bigg\langle \sum_{i\in\mathcal{R}} \big(2\eta^k_i-\eta^{k-1}_i\big) + \sum_{j\in\mathcal{B}} \big(2\eta^k_j-\eta^{k-1}_j\big), x_0^k-x_0^* \bigg\rangle\notag\\
\overset{\eqref{sc}}\le&-2\mathbb{E}\big(f_0(x_0^k)-f_0(x_0^*)\big) - \mu_0\mathbb{E}\|x_0^k-x_0^*\|^2 + \frac{\mu_0}2 \mathbb{E}\|x_0^k-x_0^*\|^2 + \frac2{\mu_0} \mathbb{E}\bigg\|\sum_{j\in\mathcal{B}} \big(2\eta^k_j-\eta^{k-1}_j\big)\bigg\|^2 + 2\mathbb{E}\langle \sum_{i\in\mathcal{R}} \big(2\eta^k_i-\eta^{k-1}_i\big) , x_0^k-x_0^*\rangle\notag\\
\le&-2\mathbb{E}\big(f_0(x_0^k)-f_0(x_0^*)\big) - \frac{\mu_0}2\mathbb{E}\|x_0^k-x_0^*\|^2 + \frac2{\mu_0}9q^2\lambda^2d + 2\mathbb{E}\langle \sum_{i\in\mathcal{R}} \big(2\eta^k_i-\eta^{k-1}_i\big) , x_0^k-x_0^*\rangle.\label{sqk-step1-2}%
\end{align}
Substituting \eqref{sqk-step1-1} and \eqref{sqk-step1-2} into \eqref{sqk-step1-0} gives
\begin{align}
\mathbb{E} \|x_0^{k+1}-x_0^*\|^2 \le&~ \mathbb{E} \|x_0^k-x_0^*\|^2\left( 1- \alpha_0^k\frac{\mu_0}2 + (\alpha_0^k)^2 2L_0^2 \right) - \alpha_0^k  2\mathbb{E}\big(f_0(x_0^k)-f_0(x_0^*)\big)+ \alpha_0^k\frac{18d}{\mu_0}\lambda^2q^2 + (\alpha_0^k)^2 2(4r+3q)^2d\lambda^2\notag\\
&~+ \alpha_0^k2\mathbb{E}\langle \sum_{i\in\mathcal{R}} \big(2\eta^k_i-\eta^{k-1}_i\big) , x_0^k-x_0^*\rangle\notag\\
\le&~ \mathbb{E} \|x_0^k-x_0^*\|^2 - 2\alpha_0^k \mathbb{E}\big(f_0(x_0^k)-f_0(x_0^*)\big)+ \alpha_0^k\frac{18d}{\mu_0}\lambda^2q^2 + (\alpha_0^k)^2 2(4r+3q)^2d\lambda^2 \label{sqk-step1}\\
&~+ \alpha_0^k2\mathbb{E}\langle \sum_{i\in\mathcal{R}} \big(2\eta^k_i-\eta^{k-1}_i\big) , x_0^k-x_0^*\rangle\notag,
\end{align}
where the last inequality comes from that $\alpha_0^k\le \bar{A}\le\frac{\mu_0}{4L_0^2}$.

{\bf Step 2.} Accordingly, at the worker side, we have for any $i \in \mathcal{R}$ that \eqref{eq:expandi} holds, as
\begin{align}\label{sqk-step2-0}
\mathbb{E} \|x_i^{k+1}-x_i^*\|^2 = &~\mathbb{E} \|x_i^k-x_i^*\|^2 + (\alpha_i^k)^2\mathbb{E}\big\| F'(x_i^k,\xi_i^k) + \big(2\eta^k_i-\eta^{k-1}_i\big) \big\|^2 - 2\alpha_i^k \mathbb{E}\big\langle F'(x_i^k,\xi_i^k) + \big(2\eta^k_i-\eta^{k-1}_i\big), x_i^k-x_i^* \big\rangle. 
\end{align}
For the second term, we still have \eqref{eq:xi-2}, which can be further bounded by
\begin{align}
&~\mathbb{E}\big\| F'(x_i^k,\xi_i^k) + \big(2\eta^k_i-\eta^{k-1}_i\big) \big\|^2 \le 2 L_i^2\mathbb{E}\big\| x_i^k -x_i^*\big\|^2 + 4\delta_i^2 + 64\lambda^2d.\label{sqk-step2-1}
\end{align}
Here we use the fact that $f_0$ has Lipschitz continuous gradients in Assumption \ref{ass2}.
Then, using $\langle F_i'(x_i^k), x_i^k-x_i^*\rangle \ge F_i(x_i^k)-F_i(x_i^*) + \frac{\mu_i}2\|x_i^k-x_i^*\|^2$ as $f_0$ is strongly convex in Assumption \ref{ass1}, we bound the third term in \eqref{sqk-step2-0} as
\begin{align}
-2\mathbb{E}\big\langle F'(x_i^k,\xi_i^k) + \big(2\eta^k_i-\eta^{k-1}_i\big), x_i^k-x_i^* \big\rangle
=&~-2\mathbb{E}\big\langle F'_i(x_i^k), x_i^k-x_i^* \big\rangle - 2\mathbb{E}\langle 2\eta^k_i-\eta^{k-1}_i, x_i^k-x_i^* \rangle\notag \\
\le&~-2\mathbb{E}\big(F_i(x_i^k) - F_i(x_i^*)\big) - \mu_i \mathbb{E}\|x_i^k-x_i^*\|^2 -2 \mathbb{E}\langle 2\eta^k_i-\eta^{k-1}_i, x_i^k-x_i^* \rangle.\label{sqk-step2-2}
\end{align}

Substituting \eqref{sqk-step2-1} and \eqref{sqk-step2-2} into \eqref{sqk-step2-0} gives
\begin{align}\label{sqk-step2}
\mathbb{E} \|x_i^{k+1}-x_i^*\|^2
\le&~\mathbb{E} \|x_i^k-x_i^*\|^2\left( 1 - \alpha_i^k\mu_i + (\alpha_i^k)^2 2L_i^2 \right) -\alpha_i^k 2\mathbb{E}\big(F_i(x_i^k) - F_i(x_i^*)\big) +(\alpha_i^k)^2\big(64\lambda^2d+ 4\delta_i^2\big) \notag\\
&~-\alpha_i^k 2 \mathbb{E}\langle 2\eta^k_i-\eta^{k-1}_i, x_i^k-x_i^* \rangle\notag\\
\le&~\mathbb{E} \|x_i^k-x_i^*\|^2 \frac{\alpha_0^{k-1}/\alpha_i^{k-1}}{\alpha_0^k/\alpha_i^k} - 2\alpha_i^k\mathbb{E}\big(F_i(x_i^k) - F_i(x_i^*)\big) +(\alpha_i^k)^2\big(64d\lambda^2+ 4\delta_i^2\big) \\
&~-\alpha_i^k 2 \mathbb{E}\langle 2\eta^k_i-\eta^{k-1}_i, x_i^k-x_i^* \rangle\notag,
\end{align}
where the last inequality holds from the bound of $A$, such that
$$\frac{\alpha_0^{k-1}/\alpha_i^{k-1}}{\alpha_0^k/\alpha_i^k} - \left( 1 - \alpha_i^k\mu_i + (\alpha_i^k)^2 2L_i^2 \right) = \alpha_i^k\big(\mu_i - \alpha_i^k 2L_i^2 - \alpha_0^{k-1} \frac{(m-1)\beta c}{\sqrt{k}+\sqrt{k-1}}\big)\ge\alpha_i^k\big(\mu_i - (2L_i^2+ (m-1)\beta c)A\big)\ge0.$$

{\bf Step 3.}
Denote $$F^k=\sum_{i\in\mathcal{R}} \mathbb{E}[F(x_i^k,\xi_i^k)]+f_0(x_0^k)=\sum_{i\in\mathcal{R}} F_i(x_i^k)+f_0(x_0^k),\quad F^* = \sum_{i\in\mathcal{R}} F_i(x_i^*)+f_0(x_0^*)=\min_{\tilde{x}} \sum_{i\in\mathcal{R}} \mathbb{E}[F({\tilde{x}},\xi_i)]+f_0({\tilde{x}}),$$ and define a Lyapunov function $\bar{V}^k=E \|x_0^k-x_0^*\|^2+\sum_{i \in \mathcal{R}} \big(\frac{\alpha_0^{k-1}}{\alpha_i^{k-1}}\mathbb{E} \|x_i^k-x_i^*\|^2+\frac{2\alpha_0^{k-1}}\beta\|\eta_i^{k-1}\|^2\big)$. From \eqref{eq:lem:eta2}, we have
\begin{align}
-2\alpha_0^k\mathbb{E}\big\langle 2\eta^k_i-\eta^{k-1}_i, (x_i^k-x_i^*)-(x_0^k-x_0^*) \big\rangle \le \frac{2\alpha_0^k}\beta \left( \|\eta_i^{k-1}\|^2 - \|\eta_i^{k}\|^2\right) \le \frac{2\alpha_0^{k-1}}\beta\|\eta_i^{k-1}\|^2 - \frac{2\alpha_0^k}\beta\|\eta_i^{k}\|^2 \label{sqk-step3-3}.
\end{align}
Consequently, combining \eqref{sqk-step1}, \eqref{sqk-step2}, and \eqref{sqk-step3-3} together gives
\begin{align}
2\alpha_0^k \big( F^k - F^*\big)
\le&~ \bar{V}^k-\bar{V}^{k+1} + \alpha_0^k\frac{18d}{\mu_0}\lambda^2q^2 + \alpha_0^k\bigg(\alpha_0^k 2(4r+3q)^2d\lambda^2 + \alpha_i^k \big(64rd\lambda^2 + 4\sum_{i\in\mathcal{R}}\delta_i^2\big) \bigg)\notag\\
\le &~\bar{V}^k-\bar{V}^{k+1} + \bar{c}_1\alpha_i^k + \bar{c}_2(\alpha_i^k)^2,
\label{sqk-step3}
\end{align}
where $\bar{c}_1=\frac{18d}{\mu_0}\lambda^2q^2$ and $\bar{c}_2=\big(2(4r+3q)^2 + 64r\big)d\lambda^2 + 4\sum_{i\in\mathcal{R}}\delta_i^2.$
Summing up \eqref{sqk-step3} from $0$ to $k$, we obtain
\begin{align}
2\sum_{l=1}^k\alpha_0^l(F^l-F^*)\le& \bar{V}^1 - \bar{V}^{k+1} + \bar{c}_1\sum_{l=1}^k\alpha_i^l + \bar{c}_2\sum_{l=1}^k(\alpha_i^l)^2\le \bar{V}^1 + \bar{c}_1\sum_{l=1}^k\alpha_i^l + \bar{c}_2\sum_{l=1}^k(\alpha_i^l)^2.
\end{align}
Dividing both sides by $2\sum_{l=1}^k\alpha_0^l$ gives
\begin{align}
\sum_{l=1}^k\frac{\alpha_0^l}{\sum_{l'=1}^k\alpha_0^{l'}}F^l-F^*\le m\bar{c}_1 + \frac{\bar{V}^1+\bar{c}_2\sum_{l=1}^k(\alpha_i^l)^2}{2\sum_{l=1}^k\alpha_0^l}\le\frac{\bar{V}^1+\frac{\bar{c}_2}{c^2}\log(k+1)}{4(c+m\beta)\sqrt{k}}.\label{sqk-step4}
\end{align}
The convexity of $F$ and $f_0$ leads to $\sum_{i\in\mathcal{R}} \mathbb{E}[F(\bar{x}_i^k,\xi_i)]+f_0(\bar{x}_0^k) - F^* \le \sum_{l=1}^k\frac{\alpha_0^l}{\sum_{l'=1}^k\alpha_0^{l'}}F^l-F^*$. Combining this inequality and \eqref{sqk-step4}, we complete the proof.
\end{proof}


\clearpage
\begin{thebibliography}{99}


\bibitem{Agrawal2000}
R. Agrawal and R. Srikant. ``Privacy-preserving Data Mining," {\it Proceedings of ACM SIGMOD}, 2000.

\bibitem{Sicari2015}
S. Sicari, A. Rizzardi, L. Grieco, and A. Coen-Porisini, ``Security, Privacy and Trust in Internet of Things: The Road Ahead,'' {\it Computer Networks}, vol. 76, pp. 146--164, 2015.

\bibitem{Lu2018}
L. Zhou, K. Yeh, G. Hancke, Z. Liu, and C. Su, ``Security and Privacy for the Industrial Internet of Things: An Overview of Approaches to Safeguarding Endpoints," {\it IEEE Signal Processing Magazine}, vol. 35, no. 5, pp. 76--87, 2018.

\bibitem{Konecny2015}
J. Konecny, H. McMahan, and D. Ramage, ``Federated Optimization: Distributed Optimization Beyond the Datacenter,'' {\it arXiv}: 1511.03575, 2015.

\bibitem{Konecny2016}
J. Konecny, H. McMahan, F. Yu, P. Richtarik, A. Suresh, and D. Bacon, ``Federated Learning: Strategies for Improving Communication Efficiency,'' {\it arXiv}: 1610.05492, 2016.

\bibitem{Kairouz2021}
P. Kairouz and H. McMahan. ``Advances and Open Problems in Federated Learning," {\it Foundations and Trends in Machine Learning}, vol. 14, no. 1, 2021.

\bibitem{Lamport1982}
L. Lamport, R. Shostak, and M. Pease, ``The Byzantine Generals Problem,'' {\it ACM Transactions on Programming Languages and Systems}, vol. 4, no. 3, pp. 382--401, 1982.

\bibitem{Lynch996}
N. Lynch, \textit{Distributed Algorithms}, Morgan Kaufmann Publishers, San Francisco, USA, 1996.

\bibitem{Venpaty2013}
A. Vempaty, L. Tong, and P. K. Varshney, ``Distributed Inference with Byzantine Data: State-of-the-Art Review on Data Falsification Attacks." {\it IEEE Signal Processing Magazine}, vol. 30, no. 5, pp. 65--75, 2013.

\bibitem{Chen2018a}
Y. Chen, S. Kar, and J. M. F. Moura, ``The Internet of Things: Secure Distributed Inference." {\it IEEE Signal Processing Magazine}, vol. 35, no. 5, pp. 64--75, 2018.

\bibitem{Blanchard2017}
P. Blanchard, E. M. E. Mhamdi, R. Guerraoui, and J. Stainer, ``Machine Learning with Adversaries: Byzantine Tolerant Gradient Descent,'' {\it Proceedings of NeurIPS}, 2017.

\bibitem{Li2019}
S. Li, Y. Cheng, W. Wang, Y. Liu, and T. Chen, ``Abnormal Client Behavior Detection in Federated Learning,'' {\it arXiv}: 1910.09933, 2019.

\bibitem{Li2020}
S. Li, Y. Cheng, W. Wang, Y. Liu, and T. Chen, ``Learning to Detect Malicious Clients for Robust Federated Learning,'' {\it arXiv}: 2002.00211, 2020.

\bibitem{Ravi2019}
N. Ravi and A. Scaglione, ``Detection and Isolation of Adversaries in Decentralized Optimization for Non-Strongly Convex Objectives,'' {\it Proceedings of IFAC Workshop on Distributed Estimation and Control in Networked Systems}, 2019.

\bibitem{Alistarh2018}
D. Alistarh, Z. Allen-Zhu, and J. Li, ``Byzantine Stochastic Gradient Descent,'' {\it Proceedings of NeuIPS}, 2018.

\bibitem{Zhu2021}
Z. Allen-Zhu, F. Ebrahimianghazani, J. Li, and D. Alistarh, ``Byzantine-Resilient Non-Convex Stochastic Gradient Descent,'' {\it Proceedings of ICLR}, 2021.

\bibitem{Chen2018}
L. Chen, H. Wang, Z. Charles, and D. Papailiopoulos, ``DRACO: Byzantine-resilient Distributed Training via Redundant Gradients,'' {\it Proceedings of ICML}, 2018.

\bibitem{Chen2018b}
Y. Chen, L. Su, and J. Xu, ``Distributed Statistical Machine Learning in Adversarial Settings: Byzantine Gradient Descent,'' {\it Proceedings of the ACM on Measurement and Analysis of Computing Systems}, 2018.

\bibitem{Xie2018a}
C. Xie, O. Koyejo, and I. Gupta, ``Generalized Byzantine Tolerant SGD,'' {\it arXiv}: 1802.10116, 2018.

\bibitem{Cao2020}
X. Cao and L. Lai, ``Distributed Approximate Newton's Method Robust to Byzantine Attackers," {\it IEEE Transactions on Signal Processing}, vol. 68, pp. 6011--6025, 2020.

\bibitem{Yin2018a}
D. Yin, Y. Chen, K. Ramchandran, and P. Bartlett, ``Byzantine-robust Distributed Learning: Towards Optimal Statistical Rates,'' {\it Proceedings of ICML}, 2018.

\bibitem{Xie2018b}
C. Xie, S. Koyejo, and I. Gupta, ``Zeno: Distributed Stochastic Gradient Descent with Suspicion-based Fault-tolerance,'' {\it Proceedings of ICML}, 2019.

\bibitem{Xie2018c}
C. Xie, O. Koyejo, and I. Gupta, ``Phocas: Dimensional Byzantine-resilient Stochastic Gradient Descent,'' {\it arXiv}: 1805.09682, 2018.

\bibitem{Mhamdi2018}
E. M. E. Mhamdi, R. Guerraoui, and S. Rouault, ``The Hidden Vulnerability of Distributed Learning in Byzantium,'' {\it Proceedings of ICML}, 2018.

\bibitem{Kevin2019}
K. Hsieh, A. Phanishayee, O. Mutlu, and P. B. Gibbons, ``The Non-IID Data Quagmire of Decentralized Machine Learning,'' {\it Proceedings of ICML}, 2020.

\bibitem{He2020-rs}
L. He, S. P. Karimireddy, and M. Jaggi, ``Byzantine-robust Learning on Heterogeneous Datasets via Resampling,'' {\it arXiv}: 2006.09365, 2020.

\bibitem{Peng2020-rs}
J. Peng, Z. Wu, Q. Ling, and T. Chen, ``Byzantine-Robust Variance-Reduced Federated Learning over Distributed Non-i.i.d. Data,'' {\it arXiv}: 2009.08161, 2020.

\bibitem{Liping2018}
L. Li, W. Xu, T. Chen, G. Giannakis, and Q. Ling, ``RSA: Byzantine-robust Stochastic Aggregation Methods for Distributed Learning from Heterogeneous Datasets,'' {\it Proceedings of AAAI}, 2019.

\bibitem{Wu2019}
Z. Wu, Q. Ling, T. Chen, and G. Giannakis, ``Federated Variance-Reduced Stochastic Gradient Descent with Robustness to Byzantine Attacks,'' {\it IEEE Transactions on Signal Processing}, vol. 68, pp. 4583--4596, 2020.

\bibitem{Mahdi2020}
E. M. E. Mhamdi, R. Guerraoui, and S. Rouault, ``Distributed Momentum for Byzantine-resilient Learning,'' {\it Proceedings of ICLR}, 2021.

\bibitem{Khanduri2019}
P. Khanduri, S. Bulusu, P. Sharma, and P. Varshney, ``Byzantine Resilient Non-Convex SVRG with Distributed Batch Gradient Computations,'' {\it arXiv}: 1912.04531, 2019.

\bibitem{Karimireddy2021}
S. P. Karimireddy, L. He, and M. Jaggi, ``Learning from History for Byzantine Robust Optimization,'' {\it Proceedings of ICML}, 2021.

\bibitem{Damaskinos2018}
G. Damaskinos, E. M. E. Mhamdi, R. Guerraoui, R. Patra, and M. Taziki, ``Asynchronous Byzantine Machine Learning (the Case of SGD),'' {\it Proceedings of ICML}, 2018.

\bibitem{Yang2020}
Y. Yang and W. Li, ``BASGD: Buffered Asynchronous SGD for Byzantine Learning,'' {\it arXiv}: 2003.00937, 2020.

\bibitem{Xie2019}
C. Xie, S. Koyejo, and I. Gupta, ``Zeno++: Robust Fully Asynchronous SGD,'' {\it Proceedings of ICML}, 2020.

\bibitem{Yin2018b}
D. Yin, Y. Chen, R. Kannan, and P. Bartlett, ``Defending Against Saddle Point Attack in Byzantine-Robust Distributed Learning,'' {\it Proceedings of ICML}, 2019.

\bibitem{Yang2019a}
Z. Yang and W. U. Bajwa, ``ByRDiE: Byzantine-Resilient Distributed Coordinate Descent for Decentralized Learning,'' {\it IEEE Transactions on Signal and Information Processing over Networks}, vol. 5, no. 4, pp. 611--627, 2019.

\bibitem{Yang2019b}
Z. Yang and W. U. Bajwa, ``BRIDGE: Byzantine-resilient Decentralized Gradient Descent,'' {\it arXiv}: 1908.08098, 2019.

\bibitem{Guo2020}
S. Guo, T. Zhang, X. Xie, L. Ma, T. Xiang, and Y. Liu, ``Towards Byzantine-resilient Learning in Decentralized Systems,'' {\it arXiv}: 2002.08569, 2020

\bibitem{Peng2021-dec}
J. Peng, W. Li, and Q. Ling, ``Byzantine-robust Decentralized Stochastic Optimization over Static and Time-varying Networks,'' {\it Signal Processing}, vol. 183, no. 108020, 2021.

\bibitem{Hua2013}
H. Ouyang, N. He, and A. Gray, ``Stochastic ADMM for Nonsmooth Optimization,'' {\it arXiv}: 1211.0632, 2012.

\bibitem{sta-opt}
W. Ben-Ameur, P. Bianchi, and J. Jakubowicz, ``Robust Distributed Consensus Using Total Variation,'' {\it IEEE Transactions on Automatic Control}, vol. 61, no. 6, pp. 1550--1564, 2016.

\bibitem{Nesterov}
Y. Nesterov, {\it Introductory Lectures on Convex Optimization: A Basic Course}, Springer, Boston, USA, 2004.


\end{thebibliography}
\end{document}